\documentclass[12pt,oneside,reqno]{amsart}

\usepackage{shortcuts,amssymb}
\usepackage{graphicx}
\usepackage{hyperref}
\usepackage{pdfsync}
\usepackage[ruled]{algorithm2e}
\hypersetup{
    colorlinks,
    citecolor=black,
    filecolor=black,
    linkcolor=black,
    urlcolor=black
}

\newcommand{\Z}{\mathbb{Z}}
\newcommand{\om}{\Omega}
\newcommand{\HH}{\mathcal{H}}
\newcommand{\DD}{\varLambda}
\newcommand{\vv}{\mathrm{v}}
\newcommand{\divg}{\mathrm{div}}
\newcommand{\grad}{\mathrm{grad}}
\newcommand{\qq}{q}

\newcommand{\dx}{\mathrm{d}x}
\newcommand{\ii}{k}
\newcommand{\jj}{l}
\newcommand{\sysize}{b}
\newcommand{\Face}{\mathcal{F}}
\renewcommand{\forall}{\text{ for all }}
\newcommand{\mat}[1]{\mathtt{{#1}}}
\newcommand{\ipm}[1]{\langle #1\rangle}
\newcommand{\zs}[1]{{#1}}

\counterwithin*{equation}{section}

\title[Stability of SAT methods for tents]
{Stability of Structure-Aware Taylor Methods for Tents}

\author[J.~Gopalakrishnan]{Jay Gopalakrishnan}
\address{Portland State University, PO Box 751, Portland, OR 97207, USA }
\email{gjay@pdx.edu}

\author[Z.~Sun]{Zheng Sun}
\address{(Corresponding author.) Department of Mathematics,
The University of Alabama,
Box 870350,
Tuscaloosa, AL 35487, USA}
\email{zsun30@ua.edu}

\begin{document}


\begin{abstract}
  Structure-aware Taylor (SAT) methods are a class of timestepping
  schemes designed for propagating linear hyperbolic solutions within
  a tent-shaped spacetime region. Tents are useful to design explicit
  time marching schemes on unstructured advancing fronts with built-in
  locally variable timestepping for arbitrary spatial and temporal
  discretization orders. The main result of this paper is that an
  $s$-stage SAT timestepping within a tent is weakly stable under the
  time step constraint $\Delta t \leq Ch^{1+1/s}$, where $\Delta t$ is
  the time step size and $h$ is the spatial mesh size.  Improved
  stability properties are also presented for high-order SAT time
  discretizations coupled with low-order spatial polynomials. A
  numerical verification of the sharpness of proven estimates is also
  included.
\end{abstract}


\maketitle

\section{Introduction}

Spacetime methods for solving evolution equations can easily
incorporate widely varying spatio-temporal grid sizes and
discretization orders.  However, to be competitive with standard
timestepping methods, spacetime methods must have memory requirements
and coupling of degrees of freedom that are comparable to standard
timestepping methods.  Such competitive spacetime methods can indeed
be constructed for hyperbolic systems by partitioning the spacetime
region into tent-shaped subregions satisfying causality: one then
propagates numerical solutions asynchronously across an unstructured
advancing front.  The process of creating a mesh of tents by advancing
spacetime fronts is referred to as ``tent pitching'' and recent
methods like the Mapped Tent Pitching (MTP)
schemes~\cite{GopalHochsSchob20, gopalakrishnan2017mapped} have proven
themselves to be competitive tent-based alternatives to standard
timestepping schemes, especially on complex geometries.  Many previous
works, both in the engineering and the mathematics literature, have
constructed tent-based numerical
methods~\cite{abedi2018spacetime,falk1999explicit,miller2008spacetime,monk2005discontinuous}.

The purpose of this paper is to provide a complete stability analysis
of the structure-aware Taylor (SAT) methods, a class of timestepping
methods suitable for the above-mentioned MTP schemes applied to linear
hyperbolic systems.  To understand the origins of the SAT
timestepping, consider the computational drawbacks that can arise from
the lack of tensor-product structure in  tent-shaped domains,
including the inability to use standard discretizations combined with
timestepping within a tent.  A proposal to overcome such difficulties
was made in~\cite{gopalakrishnan2017mapped}. The idea is to map the
non-tensor-product tent region to a tensor-product cylindrical
region. This made fully explicit timestepping within spacetime tents,
combined with a standard discontinuous Galerkin (DG) spatial
discretization, possible.
Indeed, after the semidiscretization, the
unknown function of a pseudotime variable $\hat{t}$, introduced later as
$\hat u_h(\hat{t})$, satisfies the ordinary differential equation (ODE) 
\begin{equation}
  \label{eq:MuAuODE}
  \frac{\mathrm d}{\mathrm d \hat{t}} \left(M(\hat{t}) \hat u_h\right) = A\,\hat u_h  
\end{equation}
with a time-dependent mass
operator $M(\hat{t})$ and a differential operator $A$, defined later  in~\eqref{eq-M0M1A}. Introducing
$y(\hat{t}) = M(\hat{t}) \hat u_h(\hat{t})$, this ODE can be restated as 
\begin{equation}
  \label{eq:nonauto}
  \frac{\mathrm d y}{\mathrm d \hat{t}}  = 
  A M(\hat{t})^{-1} y.  
\end{equation}
Although high-order standard Runge--Kutta timestepping can be
applied to solve~\eqref{eq:nonauto}, the resulting solutions were not
observed to achieve the expected high orders of accuracy, as reported
in~\cite{GopalHochsSchob20, gopalakrishnan2020structure}.
New timestepping methods, incorporating the structure of the
time-dependent mass matrix arising from tents, were then developed.
Specifically, the SAT timestepping was proposed
in~\cite{GopalHochsSchob20} to address this issue for linear
hyperbolic systems. Its extension to nonlinear hyperbolic systems,
named SARK timestepping, was proposed
in~\cite{gopalakrishnan2020structure}.

Energy-type stability estimates for these new timestepping schemes
remained unknown until \cite{DrakeGopalSchob21}, where a framework for 
the stability and error analysis of the MTP methods for linear
hyperbolic systems was constructed. For SAT schemes, the stability
of the first- and the second-order methods were proved in \cite{DrakeGopalSchob21}. In particular, the analysis requires a $3/2$-Courant--Friedrichs--Levy (CFL) condition for the stability of the second-order methods. Here, as in \cite{burman2010explicit}, $\eta$-CFL condition refers to the time step constraint $\Delta t \leq C h^\eta$, for some fixed constant $C$. $\Delta t$ is the time step size and $h$ is the spatial mesh size. Furthermore, based on the numerical
tests in \cite[Section 6.1]{gopalakrishnan2020structure},
extrapolating from the provable cases, it was conjectured in 
\cite{DrakeGopalSchob21} that the SAT method is stable under a
$(1+1/s)$-CFL condition for MTP schemes, where $s$ is the order of the
SAT time discretization. This more restrictive CFL condition is also
required by and known to be necessary for standard explicit Runge--Kutta methods when applied to hyperbolic equations in certain cases
\cite{burman2010explicit, zhang2004error, xu2019l2,
  xu2020superconvergence}. For the SAT methods, the rigorous stability
proof is only available for $s \leq 2$. The analysis for the higher-order methods remained open.

In this paper, we prove the above-mentioned conjecture for any $s$
through an energy-type stability analysis. Since naive eigenvalue analyses with
the stability regions may lead to insufficient and even misleading results
\cite{iserles2009first,levy1998semidiscrete,tadmor2002semidiscrete,sun2017rk4},
energy arguments are widely used for stability analysis, especially for
systems resulting from discretizations of partial differential
equations. For implicit time marching methods or dissipative
equations, universal stability analysis is well documented
\cite{butcher2016numerical, gottlieb2001strong}. However for
hyperbolic type problems with high-order explicit methods, a
systematic analysis was not available until recently. Based on the
techniques developed in the analysis of the fourth order Runge--Kutta
methods \cite{sun2017rk4,ranocha2018l_2}, in \cite{sun2019strong}, Sun
and Shu proposed a general framework on analyzing the strong stability of
explicit Runge--Kutta (Taylor) methods of arbitrary order. This work
also relates to the fully discrete analysis of Runge--Kutta DG methods
in \cite{xu2019l2,xu2020superconvergence}. We also refer to
\cite{sun2021enforcing,sun2022energy} on further extensions of
the work in \cite{sun2019strong}. Results on nonlinear or
non-autonomous problems can be found in
\cite{ranocha2021strong,ranocha2020energy}.

The main challenge in the analysis of the SAT method is to
appropriately handle the mass matrix that is affine-linear in a
pseudotime variable arising from the mapping. It leads to the
following complications that have not been encountered in the analysis
of standard Runge--Kutta (Taylor) methods. First, the numerical
dissipation will depend on the time derivative of the mass
matrix. Second, the high-order spatial derivatives are defined via a
recursive formula, rather than a simple matrix power. Third, there
are extra tail terms arising in the simplification of energy equality,
and finding an appropriate way of grouping the terms becomes an
issue. The key ingredient for solving these issues is to introduce a
novel discrete integration by parts formula for the MTP schemes (developed in Lemma~\ref{lem-ibp} below). The analysis of the SAT
method can be viewed as a generalization of the framework developed
in~\cite{sun2019strong}. In the special case of constant mass matrix,
many results in Section~\ref{sec-sat} reduce to the known 
estimates for the standard Runge--Kutta or Taylor methods.

Furthermore, we show that when a low-order spatial discretization is
coupled with a high-order SAT timestepping method, the fully discrete
scheme may exhibit improved stability properties with a relaxed CFL
condition. Consider symmetric linear hyperbolic systems with constant
coefficients, we show that with spatially piecewise constant
elements ($p = 0$), the SAT scheme is strongly stable under the usual
CFL condition for any order. When the spatial polynomial degree
satisfies $0<p\leq (s-1)/2$, then we show that the numerical method is
weakly stable under the $(1+1/(2s-2p))$-CFL condition. The key step of
the proof is to give an explicit characterization of the derivative
operator in the SAT scheme (found in Lemma~\ref{lem-lowXk} below). The
estimates are verified to be sharp within a subtent using the linear
advection equation in one dimension (in Section \ref{sec-num}). This
investigation of improved stability when employing low-order
polynomials is inspired by a similar study of the standard
Runge--Kutta DG methods for linear advection by Xu et
al. in~\cite{xu2019l2}.
It turns
out that the SAT-DG methods in this paper exhibit stability properties
that are different from those of the standard Runge--Kutta DG methods for linear
autonomous equations---for the latter, strong stability can be
achieved for $p>0$ when sufficiently high-order timestepping methods
are used. This is not the case for SAT-DG methods for
\eqref{eq:MuAuODE}, whose weak
stability properties seem more in line with those of nonautonomous
equations, which is perhaps not surprising since~\eqref{eq:nonauto} is not autonomous.

The rest of the paper is organized as follows. In Section
\ref{sec-prelim}, we briefly outline the MTP scheme and state the
corresponding ordinary differential equation (ODE) system arising from
the semidiscretization after the tent mapping. The weak stability of
the SAT method under the $(1+1/s)$-CFL condition is proved in
Section~\ref{sec-sat}. In Section~\ref{sec-low}, we prove the improved
stability properties of SAT-DG schemes with low-order spatial
polynomials. Then we show the sharpness of our estimates in
Section~\ref{sec-num} using the one-dimensional example. Proofs of all
the lemmas in these sections are presented in
Section~\ref{sec:proofs-sat} in the same order they appeared
previously.  Finally, conclusions and future work are discussed in
Section~\ref{sec-conclusion}.

\section{Tents, Maps, and the SAT timestepping}\label{sec-prelim}

In this section, we describe a model symmetric linear hyperbolic
problem and how one constructs an advancing front solution 
on unstructured meshes using spacetime tents. Here we collect
preliminary results from elsewhere that we need for the subsequent
stability analysis.

Let $\Omega \subseteq \mathbb{R}^d$ represent the spatial domain of
the simulation and let
$u = u(x,t): \Omega\times [0,t_{\max}] \to \mathbb{R}^\sysize$ be a
vector-valued function. Our goal is to solve the symmetric linear
hyperbolic system
\begin{equation}\label{eq-slhs}
\partial_t g(u) + \divg_x f(u) = 0,
\end{equation}
with 
\begin{equation}\label{eq-gf}
[g(u)]_l = \sum_{k = 1}^\sysize \mathcal{G}_{lk} u_k, \quad [f(u)]_{lj} = \sum_{k = 1}^\sysize \mathcal{L}_{lk}^{(j)} u_k, \quad l = 1,\cdots, \sysize, \quad j = 1,\cdots d.
\end{equation}
Here
$\mathcal{G} = [\mathcal{G}_{lk}]: \Omega\to
\mathbb{R}^{\sysize\times \sysize}$ and
$\mathcal{L}^{(j)} = [\mathcal{L}^{(j)}_{lk}] : \Omega \to
\mathbb{R}^{\sysize\times \sysize}$ are symmetric bounded
matrix-valued functions, and $\mathcal{G}$ is uniformly positive definite on $\overline{\Omega}$. Furthermore, let us assume that $\sum_{j=1}^d \partial_j \mathcal{L}^{(j)} = 0$ in the sense of distributions \cite[Subsection~2.1]{DrakeGopalSchob21}, so that the (weighted) $L^2$ energy of \eqref{eq-slhs} is nonincreasing in time. To avoid unnecessary technicalities, we consider periodic boundary conditions or compactly supported solutions in this paper, although more general boundary conditions
can be handled using the techniques outlined in
\cite[Subsection~2.1]{DrakeGopalSchob21}. We proceed to build a spacetime mesh of tents atop $\om$.

First, we mesh the spatial domain.  Let $\mathcal{T}$ denote a shape
regular and conforming simplicial mesh of the spatial domain
$\Omega$. Let $h$ be the mesh size parameter equaling the maximal
diameter of elements in $\mathcal T$. We march forward in time by
considering a sequence of advancing fronts
$\vphi_i: \om \to \mathbb R$, $i=0, \ldots, m$. Here
$\{\vphi_i\}_{i=0}^m$ are continuous piecewise linear functions,
specifically the lowest-order Lagrange finite element functions, on the
mesh $\mathcal{T}$.  In particular, we have $\vphi_0(x) \equiv 0$ and
$\vphi_m(x) \equiv t_{\max}$. Given a vertex $\vv$, we define $\Omega^\vv$ to
be the vertex patch which includes spatial simplices connecting to
$\vv$. We advance from $\vphi_i$ to $\vphi_{i+1}$ over $\Omega^\vv$ by
erecting a spacetime tent pole at the vertex $\vv$ and forming the
tent
\begin{equation*}
  T_i^\vv = \{(x,t):
  \;x\in \Omega^\vv,\;
  \vphi_i(x)\leq t\leq \vphi_{i+1}(x)\}.
\end{equation*}
To ensure that each spacetime tent encloses the domain of dependence of all its points, we employ  the ``causality condition''
\begin{equation}\label{eq-causa}
\|(\grad_x \vphi_i)(x)\|_{2} < \frac{1}{c_{\max}}, \quad x \in \Omega, \quad i = 0,\ldots, m,
\end{equation}
where $c_{\max}$ is a strict upper bound of the maximal hyperbolic wave speed. For a graphical illustration of the tent-pitching meshing process, we refer to \cite[Figure 1]{gopalakrishnan2017mapped}.

In MTP schemes, one maps the tents to domains which are a tensor
product of a spatial vertex patch and a ``pseudotime'' interval in order to 
gain efficiency and to allow reutilization of common spatial
discretization tools and tensor-product techniques like timestepping.
Consider a single tent  over any given vertex patch $\Omega^\vv,$
\[
T = \{(x,t): x\in \Omega^\vv, \vphi_{\mathrm{bot}}\leq t\leq \vphi_{\mathrm{top}}\}.
\]
Here $\vphi_{\mathrm{bot}}$ and $\vphi_{\mathrm{top}}$ are restrictions of $\vphi_i$ and $\vphi_{i+1}$ over $\Omega^\vv$. They are also continuous and piecewise linear. The goal is to solve \eqref{eq-slhs} locally within the tent $T$ from $t = \vphi_{\mathrm{bot}}$ to $t = \vphi_{\mathrm{top}}$ using a timestepping technique. To this end, we transform $T$ into a tensor product domain $\hat{T} = \Omega^\vv\times [0,1]$. See \cite[Figure 2]{gopalakrishnan2017mapped}. The required change of variables is given by $(x,t) = (x, \vphi(x,\hat{t}))$, where 
\begin{equation*}
 \vphi(x,\hat{t}) = (1-\hat{t})\vphi_{\mathrm{bot}}(x) + \hat{t}\vphi_{\mathrm{top}}(x) = \vphi_{\mathrm{bot}}(x) + \hat{t} \delta(x). 
\end{equation*}
Here, 
$\delta(x) = \vphi_{\mathrm{top}}(x) - \vphi_{\mathrm{bot}}(x)$ and $\hat t$ is what we referred to above as the pseudotime variable.
From the causality condition \eqref{eq-causa}, we know that $\delta \leq C h$ for some constant $C$ depending on the wavespeed.
In  \cite[Theorem 3.1]{gopalakrishnan2017mapped} it is shown that 
the transformed unknown $\hat{u}(x,\hat{t}) = u(x,t)$ solves the equation 
\begin{equation}\label{eq-mapped}
\partial_{\hat{t}}\left(g(\hat{u}) -  f(\hat{u})\grad_x\vphi\right) + \divg_x \left(\delta f(\hat{u})\right) = 0, \quad (x,\hat{t})\in \hat{T}.
\end{equation}
Hence MTP schemes proceed by first semidiscretizing~\eqref{eq-mapped}
in space and then timestepping in pseudotime.

Let us now freeze the pseudotime variable and introduce notations associated with the spatial discretization. For the spatial discretization we use the standard DG space 
\begin{equation*}
  V_h = \{v: v|_K \in [P_p(K)]^\sysize, \text{ for all }
  K \in \mathcal{T} \text{ and } K\subseteq \Omega^\vv\}.
\end{equation*}
Here $P_p(K)$ is the space of polynomials on $K$ of degree less than or equal to $p$. Let $\Face^\vv$ be the set of facets on the spatial vertex patch $\Omega^\vv$. On each facet $F$, let $\nu = [\nu_1,\cdots,\nu_d]$ denote a spatial unit normal vector, whose direction is currently irrelevant.  Across each facet $F$, we define the jump $\lb v\rb = \lim_{\veps \to 0^+} v(x + \veps \nu) - v(x - \veps \nu)$ and the average $\{v\} = \lim_{\veps \to 0^+} (v(x + \veps \nu) + v(x - \veps \nu))/2$. Furthermore, 
we introduce the following notations
\begin{align*}
	\ip{v,w} =& \sum_{K\subseteq \Omega^\vv}\int_K v\cdot w\,\dx, \quad 
	&
		\text{for vector-valued functions }
		v,w:\Omega^\vv\to  \mathbb{R}^\sysize,\\
	\ipm{v,w} =&
	\sum_{K\subseteq \Omega^\vv}\int_K\left(\sum_{i=1}^b\sum_{j=1}^d v_{ij}w_{ij}\right)\,\dx, \quad
	&\begin{array}{r}
		\text{for matrix-valued functions}\\
		v=[v_{ij}], w=[w_{ij}]: \Omega^\vv\to \mathbb{R}^{\sysize\times d},
	\end{array}\\
	\ip{v,w}_{\Face^\vv} =& \int_{\Face^\vv} v\cdot w \,\dx, \quad 
	&
 		\text{for vector-valued functions }
		v,w:\Face^\vv\to  \mathbb{R}^\sysize,
\end{align*}
and $\| \cdot \| = (\cdot, \cdot)^{1/2}$. Note that the vertex patch $\Omega^\vv$ is omitted in the notation for the $L^2(\Omega^\vv)$-norm and inner product
to lighten the notation 
since a substantial part of our analysis will be carried out on a single given $\Omega^\vv$.
Given a selfadjoint operator $B: V_h\to V_h$, let $\ip{v,w}_B = \ip{Bv,w}$, $\nm{v}_B = \sqrt{\ip{v,v}_B}$ if $B$ is positive definite, and $\snm{v}_B = \sqrt{\ip{v,v}_B}$ if $B$ is positive semidefinite.

Applying standard DG discretization techniques to \eqref{eq-mapped},
we obtain the following semidiscrete scheme: find $\hat{u}_h(\cdot, \hat t) \in V_h$
such that
\begin{equation}\label{eq-dgscheme}
\ip {\partial_{\hat{t}}[g(\hat{u}_h)-f(\hat{u}_h)\grad_x\vphi] , v} = \ipm{\delta f(\hat{u}_h),\grad_x v} + \ip{\delta {\hat{F}}^{\nu},  \lb v\rb}_{\Face^\vv}, \quad \forall v \in V_h,
\end{equation}
where the numerical flux ${\hat{F}}^{\nu}$ is given by
\begin{equation*}
{\hat{F}}^{\nu} =  \mathcal{D}\{\hat{u}_h\} - S\lb \hat{u}_h\rb,
\end{equation*}
using the matrix functions 
$\mathcal{D} = \sum_{j = 1}^d\nu_j\mathcal{L}^{(j)}$ and  $S$,  a $\sysize\times \sysize$ constant symmetric positive semidefinite stabilization matrix. Let the 
operators $M_0, M_1, A : V_h \to V_h$ be such that their action on 
any given $\hat{u}_h\in V_h$ is defined by
\begin{subequations}\label{eq-M0M1A}
\begin{align}
\ip{ M_0 \hat{u}_h, v} =& \ip{g(\hat{u}_h)-f(\hat{u}_h)\grad_x \vphi_{\mathrm{bot}}, v},\label{eq-M0M1A-M0}\\
\ip{ M_1 \hat{u}_h, v} =& \ip{f(\hat{u}_h)\grad_x\delta, v},\label{eq-M0M1A-M1}\\
\ip{A\hat{u}_h, v} =& \ipm{\delta f(\hat{u}_h),\grad_x v} + \ip{\delta {\hat{F}}^{\nu},  \lb v\rb}_{\Face^\vv},\label{eq-M0M1A-A}
\end{align}	
\end{subequations}
for all $v\in V_h$.
Furthermore, let 
$
M(\hat{t}) = M_0 - \hat{t} M_1.
$
Then the DG scheme in \eqref{eq-dgscheme} can be written as 
\begin{equation}\label{eq-ODE}
  (M\hat{u}_h)_{\hat t} = A\hat{u}_h, \quad \hat{t} \in[0,1].
\end{equation}
where we have denoted 
the derivative $\mathrm{d} (M \hat u_h) / \mathrm{d}\hat t$ by $(M\hat{u}_h)_{\hat t}$.
Note that $M_0$, $M_1$ and $A$ are independent of $\hat{t}$, while
$M = M(\hat{t})$ is an affine linear function of $\hat{t}$. Since
$\delta$ vanishes along the boundary of $\Omega^\vv$, there is no
coupling through the numerical fluxes between $\Omega^\vv$ and its
neighboring vertex patches. Hence the system \eqref{eq-ODE} is defined
locally within $\Omega^\vv$, allowing us to localize all stability
considerations.
We will need the following  properties of the above-defined operators
established in \cite[Lemmas~3.1 and 3.2]{DrakeGopalSchob21}. We remark that although \cite{DrakeGopalSchob21} additionally assumed what we state later in~\eqref{eq:assume-pwconst}, the proofs of the propositions we list in this section, found there, do not use that assumption.

\begin{proposition}
  \label{prop-assp}
  \;	
  \begin{enumerate}
  \item The operators $M_0$, $M_1$, and $M$ are selfadjoint. In
    addition, the causality condition implies that $M_0$ and $M$ are
    positive definite.
  \item The operator 
    \begin{equation}\label{eq-D}
      D :=  - (A^\top + A + M_1) \geq 0
    \end{equation}
    is positive semidefinite. Here $A^\top$ is the adjoint operator of $A$ under $\ip{\cdot,\cdot}$. 
  \end{enumerate}
\end{proposition}

The semidefiniteness of $D$ was crucially exploited in the stability
analyses of \cite{DrakeGopalSchob21} and will also be crucial in this
paper. To understand why this is important, we reproduce a simple argument
essentially contained in \cite[Lemma~3.3]{DrakeGopalSchob21}.

\begin{proposition}[Stability of the semidiscrete scheme]
  Solutions of~\eqref{eq-ODE} are stable in the weighted $L^2$-like norm
  $\nm{\cdot}_M$, specifically, 
	\begin{equation*}
	\frac{\mathrm{d}}{\mathrm{d}\hat{t}}\nm{\hat{u}_h}_M^2 \leq 0.
	\end{equation*}
\end{proposition}
\begin{proof}
  \begin{equation*}
    \begin{split}
      \frac{\mathrm{d}}{\mathrm{d}\hat{t}} \ip{M\hat{u}_h,\hat{u}_h} 
      &= 	\ip{(M\hat{u}_h)_{\hat{t}},\hat{u}_h} + \ip{M\hat{u}_h,(\hat{u}_h)_{\hat{t}}}\\
      &=\ip{(M\hat{u}_h)_{\hat{t}},\hat{u}_h} + \ip{\hat{u}_h,M(\hat{u}_h)_{\hat{t}}}\\
      &=\ip{(M\hat{u}_h)_{\hat{t}},\hat{u}_h} + \ip{\hat{u}_h,(M\hat{u}_h)_{\hat{t}}} - \ip{\hat{u}_h,M_{\hat{t}} \hat{u}_h}\\
      &=\ip{A\hat{u}_h,\hat{u}_h} + \ip{\hat{u}_h,A\hat{u}_h} + \ip{\hat{u}_h,M_1\hat{u}_h}\\
      &=\ip{(A+A^\top+M_1)\hat{u}_h,\hat{u}_h}\\
      &= - \snm{\hat{u}_h}_D^2
    \end{split}\qquad 
    \begin{split}
      \\
      \\
      \text{since } M \text{ is selfadjoint},\\
      \\
      \text{by~\eqref{eq-ODE} and } M_{\hat{t}} = - M_1,\\
      \text{by definition of }A^\top,\\
      \text{by definition of }D \text{ and } \snm{\cdot}_{D},\\
    \end{split}
  \end{equation*}
  so the result follows from~\eqref{eq-D}.
\end{proof}

Finally, we turn to the full discretization by SAT timestepping.
The SAT approximation of  \eqref{eq-ODE} at $\hat{t} = \tau$ is given by $\hat{u}_h(\tau) \approx R_s \hat{u}_h(0)$, where
\begin{gather}
  \label{eq-sat}
  R_sv = S_1v + M^{-1}M_0 S_2v, \text{ with }
  \\
  \label{eq-S1-S2}
  S_1 v = \sum_{i = 0}^{s-1} {(i!)^{-1}}X_i v, \qquad
  S_2v = ({s!})^{-1}X_sv,
\end{gather}
and $X_i$ is defined recursively by 
\begin{equation}\label{eq-Xi}
X_0 = I \quand X_i = \tau M_0^{-1}(A+i M_1)X_{i-1} \quad \text{ for } i \geq 1. 
\end{equation}
To ensure stability, we usually need $\tau$ to be sufficiently
small. Therefore, for time marching on $\hat{T}$, we will need to
divide $\hat{T}$ into several ``subtents'' and use the propagation
operator~\eqref{eq-sat} on each subtent, as we shall see later in Section \ref{sec-subtents}.

\begin{remark}
  In the special case $M_1 = 0$, we have
  \begin{equation}\label{eq-noM1}
    R_sv = \sum_{i = 0}^{s}(i!)^{-1} \tau^i \tilde{A}^i v,
  \end{equation}
  where $\tilde{A} = M_0^{-1}A$ is a negative semidefinite operator under the inner
  product $\ip{\cdot,\cdot}_{M_0}$. Thus $R_s$ reduces to a high-order
  Runge--Kutta (Taylor) operator for linear problems that has been
  analyzed in \cite{sun2019strong}. It is no surprise therefore that
  our analysis in Section \ref{sec-sat} is substantially guided by the techniques
  in~\cite{sun2019strong}. 
  In addition, when $\tilde{A}$ in \eqref{eq-noM1} represents the DG operator for linear advection, the energy estimate of \eqref{eq-noM1} is essentially the stability estimate of the standard Runge--Kutta DG methods, which has been systematically studied by Xu et al. in~\cite{xu2019l2,xu2020superconvergence}. Some parts of our analyses are also inspired by their work, especially the improved estimate with low-order polynomials in Section \ref{sec-low}. 
\end{remark}

For analyzing $R_s$, we need bounds on the norms of the various
operators that go into building $R_s$. The following bounds are
gathered from~\cite[Lemmas 3.1 and~4.4]{DrakeGopalSchob21}.

\begin{proposition}\label{prop-opest-norm}
  There is a $C_1>0$ independent of $h$ such that 
\[
	\max \left(\nm{M_0},\nm{M_0^{-1}},\nm{M_1},\nm{M},\nm{M^{-1}},\nm{A}, \nm{D}\right)\leq C_1.
\]
\end{proposition}

\begin{proposition}\label{prop-xivest-norm-rmd}
  There is an $h$-independent constant $C_2>0$ such that for
  any $j\geq i$,
  \begin{align}
    \nm{X_j v}_{M_0}& \leq C_2 \tau^{j-i}\nm{X_iv}_{M_0}, \label{eq-xivest-norm}\\
    \snm{X_j v}_{\tau D}& \leq C_2 \tau^{j-i + \hf}\nm{X_iv}_{M_0},\label{eq-xivest-snm}
    \\ \label{eq-S2est}
    \tau \ip{M_0S_2v,M^{-1}M_1S_2v} & \leq C_2 \tau^{2s+1}\nm{v}_{M_0}^2. 
  \end{align}
\end{proposition}
\begin{proof}
  This follows immediately from the recursive definition of $X_i$ in \eqref{eq-Xi}.
\end{proof}

This completes our review of the tent-based discretization whose
stability we now proceed to analyze.

\section{Stability analysis}\label{sec-sat}

The goal of our energy-type stability analysis is to show that
$\nm{R_sv}_M$ is appropriately bounded by $\nm{v}_{M_0}$. We first obtain a bound on $\nm{R_sv}_M$ in terms of $\nm{v}_{M_0}$ and the pseudotime $\tau$ in Theorem~\ref{thm-sat1}. This then leads to the identification of a CFL condition and the main stability result of this section, Theorem~\ref{thm:summary}.

Before proceeding, let us remark an important consequence of the stability estimates.  
As shown in~\cite{DrakeGopalSchob21}---see also Remark~\ref{rmk-errorest} below---if we define the ``energy'' at the tent's
top and bottom as $\nm{R_sv}_M$ and $\nm{v}_{M_0}$, respectively, then
one can combine such  stability bounds with local truncation
error estimates to obtain bounds for the global error at the final
time, even on unstructured meshes.

\subsection{Key ideas of the analysis}

Our proof of the above-mentioned Theorem~\ref{thm-sat1} requires a
number of quite technical steps. To ease entry into these
technicalities, we identify and motivate the key ideas as lemmas here,
whose proofs are postponed to Section~\ref{sec:proofs-sat}. Using the
lemmas, we can prove Theorem~\ref{thm-sat1} at the end of
this subsection.

To consider how we may bound $\nm{R_sv}_M$ by $\nm{v}_{M_0}$ for any
$v \in V_h$, we begin by squaring both sides of \eqref{eq-sat}. Since
$M$ and $M_0$ are selfadjoint, obvious manipulations yield
\begin{align*}
\nm{R_sv}_M^2 =& \ip{MR_sv,R_sv} \\
=&\ip{MS_1v+M_0S_2v,S_1v+M^{-1}M_0S_2v}\\
=& \nm{S_1v}_M^2 + 2\ip{S_1v,S_2v}_{M_0} + \ip{M_0S_2v,M^{-1}M_0S_2v}.
\end{align*}
Since $\nm{S_1v}_M^2 = \nm{S_1v}_{M_0}^2 - \ip{S_1v,S_1v}_{\tau M_1}$, 
\begin{align}
  \nonumber
  \nm{R_sv}_M^2
  & =  \nm{S_1v}_{M_0}^2 + 2\ip{S_1v,S_2v}_{M_0} + \nm{S_2v}_{M_0}^2
  \\ \nonumber
  & \quad + \ip{M_0S_2v,(M^{-1}M_0-I)S_2v} - \ip{S_1v,S_1v}_{\tau M_1}\\
  \nonumber
  &= \nm{S_1v+S_2v}_{M_0}^2+ \ip{M_0S_2v,(M^{-1}(M_0-M))S_2v} - \ip{S_1v,S_1v}_{\tau M_1}\\ \label{eq-expRs}
  &= \Big\|\sum_{i = 0}^s (i!)^{-1}X_iv\Big\|_{M_0}^2 - \ip{S_1v,S_1v}_{\tau M_1}
    + \tau \ip{M_0S_2v,M^{-1}M_1S_2v}.
\end{align}
For ease of notation, we introduce 
\[
  F_{ij} = (X_i v, X_j v)_{\tau M_1},\quad
  G_{ij} = (X_i v, X_j v)_{M_0},\quad \text{and}  \quad 
  H_{ij} = (X_i v, X_j v)_{\tau D}.\]
These terms can be thought of as entries of symmetric matrices $F, G,$ and $H$. Moreover, the diagonal entries of $G$ and $H$ are non-negative. Using $F$ and $G$, we rewrite \eqref{eq-expRs} as
\begin{equation}
  \label{eq:1}
  \nm{R_sv}_M^2
  =
  \sum_{i,j=0}^s   \frac{G_{ij}}{{i!j!}}
  -
  \sum_{i,j=0}^{s-1}   \frac{F_{ij}}{{i!j!}}
  + \tau \ip{M_0S_2v,M^{-1}M_1S_2v}.
\end{equation}
From \eqref{eq-S2est}, it is clear that the last term is a
high-order term in $\tau$. The remaining lower-order terms above must
be carefully sorted out to obtain a stability estimate.

To this end, a critical observation is that the off-diagonal entries
$G_{ij}$ for $j>i$ can be expressed in terms of closer-to-diagonal
entries of $G, F,$ and $H$, as stated next.

\begin{lemma}
  \label{lem-ibp}
  We have
  \begin{subequations}    
    \begin{align}
      \label{eq:Gi,j+1}
    G_{ij}
    & = -\hf H_{ii} + \Big(i + \hf\Big)F_{ii},
    && \text{ if } j = i+1,
    \\ \label{eq:Gi,j+more}
    G_{ij}
    & =
      -G_{i+1, j-1} - H_{i, j-1} + (i+j) F_{i, j-1},
    && \text{ if } j >i+1.
    \end{align}
  \end{subequations}
\end{lemma}

We give a short proof in Section~\ref{sec:proofs-sat}, which is
simple, but obscures the origins of such identities. It is
illustrative to draw an analogy with
the (non-tent) case of $M_1=0$ and $X_i = (\tau M^{-1}A)^i$ is
an approximation of $(\tau \partial_x)^i$, which corresponds to the
special case when \eqref{eq-slhs} represents one-dimensional
transport. Then
$G_{ij} = \ip{X_iv,X_jv}_{M_0} \approx \tau^{i+j}\ip{\partial_x^i
  v,\partial_x^j v}$ can be manipulated by integration by parts to
obtain identities like that of the lemma. We may therefore view the
identities of Lemma~\ref{lem-ibp} as having originated in some
discrete analog of integration by parts.

One can apply Lemma~\ref{lem-ibp} recursively to simplify the first
sum of~\eqref{eq:1}. Indeed, an even more general sum can
be rearranged as stated in the next lemma.

\begin{lemma}
  \label{lem-exps}
  For any numbers $\alpha_{ij}$ with $\alpha_{ij} = \alpha_{ji}$,  the identity 
  \begin{equation}\label{eq-exp}
    \sum_{i,j=0}^s \alpha_{ij} G_{ij}
    = \sum_{i = 0}^s \beta_i G_{ii}
    + \sum_{i,j=0}^{s-1}\gamma_{ij}H_{ij}
    + \sum_{i,j=0}^{s-1}\delta_{ij}F_{ij}
  \end{equation}
  holds with
  \begin{subequations}
    \label{eq:beta-gamma-delta}
    \begin{align}
      \beta_i &= \sum_{\qq = \max\{0,2i-s\}}^{\min\{2i,s\}}\alpha_{\qq,2i-\qq}(-1)^{i-\qq},\label{eq-abeta}
      \\
      \label{eq-agamma}
      \gamma_{ij} &= \sum_{\qq=\max\{0,i+j+1-s\}}^{\min\{i,j\}} (-1)^{\min\{i,j\}+1-\qq}\alpha_{\qq,i+j+1-\qq},\\ 
      \delta_{ij} &= \sum_{\qq=\max\{0,i+j+1-s\}}^{\min\{i,j\}} (-1)^{\min\{i,j\}-\qq}\alpha_{\qq,i+j+1-\qq}(i+j+1). \label{eq-adelta}
  \end{align}
\end{subequations}
\end{lemma}

When applying Lemma~\ref{lem-exps} to treat the first sum
of~\eqref{eq:1}, the case of interest is $\alpha_{ij} =
(i!j!)^{-1}$. In this case, by a few combinatorial identities, we
obtain the following explicit expressions for some of the coefficients introduced in Lemma~\ref{lem-exps}.

\begin{lemma}\label{lem-bgd}
	When $\alpha_{ij} = (i!j!)^{-1}$, 
	\begin{align}
          &\beta_0 = 1 \quand \beta_i = 0, 
          && \text{ for }
             1 \leq i \leq s/2,\label{eq-ijbeta}
          \\
          &\gamma_{ij} = - (i!j!(i+j+1))^{-1},
          &&\text{ for } i+j\leq {s-1},\label{eq-ijgamma}
          \\
          &\delta_{ij} = (i!j!)^{-1}, \quad
          &&\text{ for } i+j \leq s-1. \label{eq-ijdelta}
	\end{align}
\end{lemma}

To motivate the next result, first
substitute~\eqref{eq-exp} into~\eqref{eq:1} to get 
\begin{equation}\label{eq-rsv-norm}
  \nm{R_sv}_{M}^2 =\sum_{i = 0}^s \beta_i G_{ii}
  + \sum_{i,j=0}^{s-1}\gamma_{ij} H_{ij}
  + \sum_{i,j=0}^{s-1}\tilde{\delta}_{ij}F_{ij}
  + \tau \ip{M_0S_2v,M^{-1}M_1S_2v},
\end{equation}
where $\tilde{\delta}_{ij} = \delta_{ij} - (i!j!)^{-1}$.  A number
of terms in the first and last sums are zero by virtue
of~\eqref{eq-ijbeta} and~\eqref{eq-ijdelta}, respectively. One might
also anticipate from~\eqref{eq-ijgamma} that a partial sum of the
second sum in~\eqref{eq-rsv-norm}  is negative. Keeping these considerations in
view, we introduce the following definition of critical indices to 
ease the bookkeeping.

\begin{definition}\label{def:crit_ind}
  \label{defn-indices}
  Let
  \begin{enumerate}
  \item $\zeta \le s$ be
    the positive integer such that $\beta_{\zeta} \neq 0$ and $\beta_i = 0$ for all $1\leq i <\zeta$,
  \item $\rho\le s$ be the largest integer such that the
    $\rho \times \rho$ principal submatrix
    $\Gamma_{\rho} = (\gamma_{ij})_{0\leq i,j\leq \rho -1}$
    is negative definite,
    
  \item $\sigma \le 2 s$ be the largest integer such that
    $\tilde{\delta}_{ij} = 0$ for all $i + j \leq
    \sigma$,  and 
  \item $\kappa = \min(2\zeta,2\rho+1,\sigma+2)$. 
  \end{enumerate}
\end{definition}  

Explicit expressions or estimates are obtained for
$\zeta, \rho, \sigma$, and $\kappa$ in the case
$\alpha_{ij} = (i!j!)^{-1}$ in the next result.  We also list the
numerical values of $\zeta$, $\rho$, $\sigma$ and $\kappa$ for
$1\leq s\leq 20$ in Table~\ref{tab-zrs}.

\begin{lemma}\label{lem-zrsk}
	When $\alpha_{ij} = (i!j!)^{-1},$
	\begin{equation}
          \label{eq-zrsk}
          \zeta = \lfloor s/2\rfloor + 1, \quad \rho \geq \lfloor(s+1)/2\rfloor, \quad \sigma = s - 1, \quand \kappa = s+1. 
	\end{equation}
\end{lemma}

\begin{table}
  \centering 
    \begin{tabular}{c|cccccccccccccccccccc}
      \hline
      $s$&1&2&3&4&5&6&7&8&9&10&11&12&13&14&15&16&17&18&19&20\\
      \hline
      $\zeta$&1&2&2&3&3&4&4&5&5&6&6&7&7&8&8&9&9&10&10&11\\
      $\rho$&1&2&2&2&3&4&4&4&5&6&6&6&7&8&8&8&9&10&10&10\\
      $\sigma$&0&1&2&3&4&5&6&7&8&9&10&11&12&13&14&15&16&17&18&19\\
      $\kappa$&2&3&4&5&6&7&8&9&10&11&12&13&14&15&16&17&18&19&20&21\\		
      \hline
    \end{tabular}
    \caption{Values of $\zeta$, $\rho$, $\sigma$ and $\kappa$ in Definition~\ref{defn-indices} for some $s$. \label{tab-zrs}}
\end{table}

In the rest of the paper, we use $\tau_0$ and $C$ (with or without 
subscripts) to denote a constant that is independent of
$h$ and $\tau$, but generally dependent on the order of SAT method $s$, the
polynomial degree $p$, the mesh regularity constant, the norm of
$g(u)$ and $f(u)$ in \eqref{eq-gf}, the constant in the causality
condition $c_{\max}$, etc. The same symbol $C$ may represent different
values at different places. We will also extensively use the fact
$\tau \le 1$ to simplify our estimates.

Now, consider how we might attempt to bound the right hand side of
\eqref{eq-rsv-norm}.  The next result, Lemma~\ref{lem-est}, which is proved using Lemma~\ref{lem-zrsk},  tells us which low-order
terms can be ignored while doing so. 
Proposition~\ref{prop-xivest-norm-rmd} tells us how the remaining high-order terms in \eqref{eq-rsv-norm} can be bounded by low-order ones.
These ideas complete the analysis as shown next.

\begin{lemma}\label{lem-est}
  There exists positive constants $\tau_0$, $C_{\beta,+}$,
  $C_{\gamma,+}$ and $C_{\delta,+}$, and a negative constant
  $C_{\gamma,-}$, such that for all $\tau \leq \tau_0$,
  \begin{align}
    \label{eq-beta}
    \sum_{i = 0}^s \beta_i G_{ii}
    & \le \beta_0 G_{00}
      +
      \left(\beta_{\zeta} + C_{\beta,+} \tau\right)
      G_{\zeta\zeta},
    \\
    \label{eq-estgamma}
    \sum_{i,j=0}^{s-1}\gamma_{ij}H_{ij}
    & \le
      {C}_{\gamma,+}\tau
      G_{\rho\rho}
      +C_{\gamma,-}\sum_{\jj=0}^{\rho-1}H_{ll},
    \\
    \label{eq-estdelta}
    \sum_{i,j = 0}^{s-1}\tilde{\delta}_{ij}F_{ij}
    & \le  C_{\delta, +}\tau^{\sigma+2}G_{00}.
  \end{align}

\end{lemma}

\begin{theorem}\label{thm-sat1}
	There exists a constant $\tau_0$ such that for all $\tau \leq \tau_0$, we have 
	\begin{equation}\label{eq-Rsvest}
          \nm{R_s v}_{M}  \leq \left(1 + C\tau^{s+1}\right)\nm{v}_{M_0}, \quad
          \text{ for all } v\in V_h.
	\end{equation}
\end{theorem}
\begin{proof}
  Using the estimates of Lemma~\ref{lem-est} in~\eqref{eq-rsv-norm},
  \begin{equation}
    \label{eq:2}
  \begin{aligned}
    \| R_s v \|_M^2
    & \le
    (1 + C_{\delta, +}\tau^{\sigma+2}) \| v \|_{M_0}^2
      + \left(\beta_{\zeta} + C_{\beta,+} \tau\right) G_{\zeta\zeta}
      + {C}_{\gamma,+}\tau     G_{\rho\rho}
    \\
    & \quad       + \tau \ip{M_0S_2v,M^{-1}M_1S_2v}.
  \end{aligned}    
  \end{equation}
  Here we have used that $C_{\gamma, -}<0$, $\beta_0=1$ (by
  Lemma~\ref{lem-bgd}), and $G_{00} = \| v \|_{M_0}^2$.  Next, we use
  a consequence of Proposition~\ref{prop-xivest-norm-rmd},
  $ G_{ii} = \nm{X_i v}_{M_0}^2\leq C \tau^{2 i}\nm{v}_{M_0}^2,$ for
  indices $i=\zeta$ and $\rho$ in~\eqref{eq:2}.  The result, when
  combined with an application of \eqref{eq-S2est}, yields
  \begin{equation}\label{eq-Rsestsq}
    \nm{R_s v}_{M}^2\leq  \left(1 + C\tau^{\min(2\zeta,2\rho+1,\sigma+2,2s+1)}\right)\nm{v}_{M_0}^2.
  \end{equation}
  By Lemma~\ref{lem-zrsk},
  $\kappa = \min(2\zeta,2\rho+1,\sigma+2) = s+1.$ Hence the theorem
  follows after taking the square root on both sides in
  \eqref{eq-Rsestsq}.
\end{proof}

\begin{remark}
An equivalent way of stating~\eqref{eq-Rsvest} is via the 
following operator norm of $R_s$
\begin{equation}\label{eq-Rsnorm}
  \nm{R_s}_{L(M_0,M)} := \sup_{0 \neq v\in V_h}\frac{\nm{R_s v}_M}{\nm{v}_{M_0}}.
\end{equation}
Clearly, \eqref{eq-Rsvest} is equivalent to
$\nm{R_s}_{L(M_0,M)} \leq 1 + C\tau^{s+1}.$  
\end{remark}

\subsection{Subtents obeying a CFL condition}\label{sec-subtents}

Theorem~\ref{thm-sat1} allows us to identify a CFL condition and a
practical subdivision of the range of pseudotime that theoretically
guarantees weak stability, as we shall now see.
We  divide the reference tent into $r$ subtents, with 
\begin{equation*}
T_{[l]} = \{(x,t): x\in \Omega^\vv, \vphi(x,\hat{t}^{[l]})\leq t\leq \vphi(x,\hat{t}^{[l+1]})\},\quad  l = 1,\cdots, r. 
\end{equation*}
Here $\hat{t}^{[l]} = (l-1)/r$. The time step size for each subtent is $\tau = r^{-1}$. In the $l$th subtent, the propagator is defined as 
\begin{equation*}
R_{[l],s}v = \sum_{k = 0}^{s-1} {(k!)^{-1}}X_k^{[l]} v + (M^{[l]})^{-1}M_0^{[l]}({s!})^{-1}X_s^{[l]} v.
\end{equation*}
Here $M_0^{[l]} = M_0 - (l-1)\tau M_1$, $M^{[l]} = M_0^{[l]} - \tau M_1$ and 
\begin{equation*}
X_0^{[l]} = I, \quand X_i^{[l]} = \tau (M_0^{[l]})^{-1}(A+kM_1)X_{i-1}^{[l]}, \quad \forall i \geq 1.
\end{equation*}
The final solution operator at $\hat{t} = 1$ is given by 
\begin{equation*}
R_{r,s} = R_{[r],s}\circ R_{[r-1],s}\circ\cdots \circ R_{[1],s}.
\end{equation*} 

\begin{theorem}\label{thm-satstab}
	If $\tau = r^{-1} \leq C h^{1/s}$ for $C$ sufficiently small, then
	\begin{equation}\label{eq-Rrsstable}
	\nm{R_{r,s}v}_{M(1)} \leq (1+ Ch)\nm{v}_{M_0}.
	\end{equation}
\end{theorem}
\begin{proof}
  Note that $M_0^{[l]}$, $M^{[l]}$ and $X_{j}^{[l]}$ still satisfy
  Propositions \ref{prop-opest-norm} and \ref{prop-xivest-norm-rmd}, hence
  the estimate \eqref{eq-Rsvest} of Theorem~\ref{thm-sat1} holds after
  replacing $R_s$ with $R_{[l],s}$, namely
  $\nm{R_{[l],s}}_{L(M^{[l-1]},M^{[l]})}\leq 1+C\tau^{s+1}$, for all
  $1\leq l \leq r$. As a result, for all $v \in V_h$, employing an 
  analog of the operator norm in~\eqref{eq-Rsnorm}, we have
  \[
    \begin{split}
      \nm{R_{r,s}v}_{M(1)}
      & \le \nm{R_{[r],s}}_{L(M^{[r-1]},M^{[r]})} \cdots \nm{R_{[1],s}}_{L(M^{[0]},M^{[1]})}\nm{v}_{M_0} \\
      & \le   \left(1 + C\tau^{s+1}\right)^{r}\nm{v}_{M_0} \\
      & \le  \left(1+C\tau^{s+1}r\right)\nm{v}_{M_0}\\ 
       & \le (1+C\tau^{s})\nm{v}_{M_0} \\
       & \le  (1+Ch)\nm{v}_{M_0}. 
    \end{split}
  \]
  Here we have used the fact $\tau = r^{-1}$ in the second last inequality and $\tau \leq C h^{1/s}$ in the last inequality. 
\end{proof}

Recall that $\delta \leq C h$. In the physical domain, the constraint $\tau \leq C h^{1/s}$ for the pseudotime coordinate should be interpreted as $\Delta t = \tau \delta \leq C h^{1 + 1/s}$, which leads to the following summary of the main result we have proven.

\begin{theorem}
  \label{thm:summary}
  The SAT timestepping for the hyperbolic equation \eqref{eq-slhs} is
  weakly stable---in the sense of \eqref{eq-Rrsstable}---under the
  $(1+1/s)$-CFL condition $\Delta t \leq C h^{1+1/s}$ whenever a
  spatial discretization satisfying the conclusions of 
  Propositions~\ref{prop-assp} and~\ref{prop-opest-norm} is used.
\end{theorem}

\begin{remark}
From \eqref{eq-Rrsstable}, it can be deduced that 
\begin{equation*}
\nm{\hat{u}_h^m}_{L^2(\Omega)} \leq \exp{(Ct_{\max})}\nm{\hat{u}_h^0}_{L^2(\Omega)}, \quad t_{\max} = m\Delta t. 
\end{equation*}
In other words, the $L^2$ norm of the solution at the final time is
bounded by a scalar multiple of that of the initial
data---see~\cite[Remark 3.14]{DrakeGopalSchob21} for further details.
\end{remark}

\begin{remark}\label{rmk-errorest}
  Based on the stability result, one can prove a high-order error estimate for the fully discrete SAT-DG scheme for linear
  hyperbolic systems---see~\cite[Theorem~4.19]{DrakeGopalSchob21}.
\end{remark}

\section{Improved stability with low-order elements}\label{sec-low}

In this section, we study the improved stability properties when a high-order SAT method is coupled with a low-order DG spatial discretization. We now  proceed under the additional assumption that
\begin{equation}
  \label{eq:assume-pwconst}
  \text{$\mathcal{G}$ and $\mathcal{L}^{(j)}$ are constant on each mesh element }
    K \in \mathcal{T}.
\end{equation}

\subsection{Key ideas for the low-order case}\label{sec-estlowXk}

Again, the analysis is based on the identity~\eqref{eq-rsv-norm}. This
time however, we will observe that many terms there can simply be
bounded, in the low-order case, by the dissipation terms
$H_{ii} = |X_i v|_{\tau D}^2$, resulting in improved stability.

The first idea towards making this precise is the identification of a
high-order spatial derivative term in $X_i$. To define this derivative
term, first note that under the assumption~\eqref{eq:assume-pwconst},
$M_0$ and $M_1$ reduce to point-wise linear operators
\begin{align}
  \label{eq-lowM0M1}
M_0w   = g(w) - f(w)\, \grad_x\vphi_{\mathrm{bot}}  ,\qquad
  M_1w = f(w) \,\grad_x \delta,
\end{align}
for any $w \in V_h$. 
Furthermore, using integration by parts in \eqref{eq-M0M1A-A},
as in \cite[Lemma~3.2]{DrakeGopalSchob21}, 
it can be verified that for all $v, w \in V_h$, 
\begin{equation*}
\begin{aligned}
\ip{A{w}, v} &= -\ip{ \divg_x(\delta f({w})), v} - \ip{\delta \mathcal{D}\lb {w} \rb , \{v \}}_{\Face^\vv}  - \ip{\delta S\lb {w} \rb, \lb v \rb}_{\Face^\vv} . 
\end{aligned}
\end{equation*}
Here and throughout, differential operators like $\divg_x$ and $\grad_x$ above, are
applied element by element.  By the product rule on each element, we
see that
$\divg_x(\delta f(w)) = \delta\divg_x f(w) + f(w) \grad_x\delta$ is in
$P_p(K)^\sysize$ for any $w \in V_h$ due to~\eqref{eq:assume-pwconst}. Hence, letting
$A_1w = - \delta\divg_x f(w)$, and defining a lifting operator $L$
by
$ \ip{ L w, v} := \ip{ \delta \mathcal{D}\lb w\rb, \{v\}}_{\Face^\vv}
+ \ip{\delta S\lb w \rb,\lb v\rb}_{\Face^\vv}$ for all
$ v, w \in V_h, $
we can rewrite $A$ as a sum of three linear operators, 
\begin{equation}
  \label{eq-lowA}
  A = A_1 - M_1 - L.
\end{equation}
Let us define the operator $K$ such that \[Kw = -\tau M_0^{-1}\divg_x f(w)\] for any $w\in V_h$. On each
element, $Kw$ has one degree less than~$w$.
The next lemma rewrites the $X_k$ defined
in~\eqref{eq-Xi} using powers of $K$, which represent higher-order spatial derivative operators. As before, all lemmas are  proved in Section~\ref{sec:proofs-sat}.

\begin{lemma}\label{lem-lowXk}
	For all $i \geq 0$, we have
	\begin{equation}\label{eq-lowXk}
	X_i =  \delta^i K^i + Z_i,
	\end{equation}
	where $Z_i$ is defined recursively by
	\begin{equation}\label{eq-Zi}
	Z_0 = 0  \quand Z_i = \tau M_0^{-1}(A+iM_1{ + }L)Z_{i-1} { - } \tau  M_0^{-1} L X_{i-1}.
	\end{equation}
\end{lemma}

The message of Lemma~\ref{lem-lowXk} is that $X_i$ can be decomposed
as a scalar multiple of a high-order spatial derivative (namely $K^i$)
plus $Z_i$. When $i\geq p+1$, we have $K^i v = 0$ and $X_i =
Z_i$.

The next key idea is that the norm of $Z_i$ can be bounded by the sum
of $H_{\jj \jj}$, which arises from the dissipation due to DG jumps, as shown in the next lemma. When combined with Lemma~\ref{lem-lowXk}
and the observation that $X_i = Z_i$ for $i\ge p+1$, this then yields
bounds for some of the numbers $G_{ii}$ and $F_{ij}$ in the subsequent result,
Lemma~\ref{lem-low}.

\begin{lemma}\label{lem-Zj}
	For all $i \geq 0$, we have
	\begin{equation}\label{eq-Zj}
	\nm{Z_iv}_{M_0}^2\leq C\tau \sum_{l = 0}^{i-1}H_{\jj \jj}.  
	\end{equation}
\end{lemma}
\begin{lemma}\label{lem-low}
	For all $i\geq p+1$, we have 
	\begin{align}
	G_{ii}&\leq C \tau \sum_{\jj =0}^{p}H_{\jj \jj},\label{eq-lowXj}\\
	\left|F_{ij}\right| & \leq C\veps^{-1}\tau^{2(i+j-p)+1}G_{00} + \veps\sum_{\jj =0}^{p}H_{\jj \jj}, \quad \forall \veps>0.\label{eq-lowXij}
	\end{align}
\end{lemma}

It now only remains to apply the above inequalities in
Lemma~\ref{lem-est} and use the resulting bounds in
\eqref{eq-rsv-norm} to obtain improved stability estimates.  We
proceed to discuss this separately for $p = 0$ in
Subsection~\ref{sec-lowXk0} and for $0<p\leq (s-1)/2$ in
Subsection~\ref{sec-lowXkp}.

\subsection{Strong stability for the lowest-order case}
\label{sec-lowXk0}

In this section, we show that when the DG spatial discretization is
used with $p = 0$, the SAT scheme is strongly stable under the usual
CFL condition for any temporal order $s\ge 1$.

\begin{theorem}\label{thm-strongstab-onesub}
  If $p = 0$, then
  there exists a constant $\tau_0$ such that for all
  $\tau \leq \tau_0$, we have 
	\begin{equation}\label{eq-strong}
	\nm{R_s v}_M\leq \nm{v}_{M_0}, \quad \forall v\in V_h. 
	\end{equation}
\end{theorem}
\begin{proof}
  The proof proceeds by  bounding the terms in the
  identity~\eqref{eq-rsv-norm}.
  By Proposition \ref{prop-xivest-norm-rmd},
  \begin{equation}
    \label{eq:Gcrudebd}
    G_{jj} = \nm{X_j v}_{M_0}^2 \leq C \nm{X_1 v}_{M_0}^2 = C G_{11}
    \quad \text{ for all } j\geq 1.      
  \end{equation}
  Hence by Lemma~\ref{lem-est}, 
  \begin{equation}\label{eq-p01}
    \sum_{i = 0}^s\beta_i G_{ii} + \sum_{i, j = 0}^{s-1}\gamma_{ij} H_{ij}\leq \beta_0 G_{00}  + CG_{11} + C_{\gamma, -} H_{00},
  \end{equation}
  where we have used the fact that $C_{\gamma,-}<0$ and
  $H_{\jj \jj}\geq0$ to drop the high-order $H_{\jj \jj}$ terms.
  Again,  by Proposition~\ref{prop-xivest-norm-rmd},
  $|F_{ij}| \le C G_{ii}^{1/2} G_{jj}^{1/2}$, which when combined with~\eqref{eq:Gcrudebd},
  yields
  \begin{equation}\label{eq-p02}
    \sum_{i,j = 0}^{s-1} \tilde{\delta}_{ij} F_{ij} + \tau\ip{M_0S_2 v, M^{-1}M_1 S_2 v} \leq CG_{11}.
  \end{equation}
  Here we have used the fact that 
  $\tilde{\delta}_{i0} = \tilde{\delta}_{0j}=0$ for all $0 \leq i,j\leq s-1$ (recall that $\sigma = s-1$)
  and $S_2v = X_s v/s!$ (see~\eqref{eq-S1-S2}) with $s\ge 1$.
  Using \eqref{eq-p01} and \eqref{eq-p02} in \eqref{eq-rsv-norm} and
  recalling that $\beta_0 G_{00} = \nm{v}_{M_0}^2$, we get
  \begin{equation*}
    \nm{R_sv}_{M}^2	\leq \nm{v}_{M_0}^2 + CG_{11} + C_{\gamma,-}H_{00}.
  \end{equation*}
  Applying  \eqref{eq-lowXj} of Lemma~\ref{lem-low} with $p=0$, we have
  $G_{11}\leq C \tau H_{00}$, and hence
  \begin{equation*}
    \nm{R_sv}_{M}^2	\leq  \nm{v}_{M_0}^2  +\left(C_{\gamma,-}+C\tau \right)H_{00}.
  \end{equation*} 
  Since $C_{\gamma,-}<0$, 
  by taking
  $\tau$ sufficiently small, we obtain~\eqref{eq-strong}.
\end{proof}

Repeatedly applying Theorem~\ref{thm-strongstab-onesub} on successive
subtents, we obtain the following analogue of Theorem~\ref{thm-satstab}.

\begin{theorem}\label{thm-strong}
  If $p = 0$,  then for any $s$, there exists a constant $\tau_0$ such that when $\tau = r^{-1}\leq \tau_0$, we have 
  \begin{equation*}
    \nm{R_{r,s} v}_{M(1)} \leq \nm{v}_{M_0}, \quad \forall v\in V_h. 
  \end{equation*}
  As a result, the SAT-DG0 scheme is strongly stable under the usual
  CFL condition $\Delta t \leq Ch$.
\end{theorem}

\subsection{Improved weak stability for other low-order cases}
\label{sec-lowXkp}

We now prove a better weak stability result for lower-order DG
discretizations beyond the lowest-order case.  The improvement is
visible when comparing the powers of $\tau$ in \eqref{eq-Rsvest} and~\eqref{eq-Rsvest-low},  and the
consequent less restrictive CFL condition in
Theorem~\ref{thm-impweak}.

\begin{theorem}\label{thm-impstab-onesub}
  If $0<p\leq (s-1)/2$, then there exists a constant $\tau_0$, such
  that for all $\tau \leq \tau_0$, 
\begin{equation}\label{eq-Rsvest-low}
\nm{R_s v}_M \leq (1+C\tau^{2s-2p+1})\nm{v}_{M_0}, \quad \quad \forall v \in V_h. 
\end{equation}
\end{theorem}

\begin{proof}
  By~\eqref{eq-zrsk} of Lemma~\ref{lem-zrsk}, the given condition on
  $p$ implies that $\zeta = \lfloor s/2\rfloor + 1 \ge p+1$ and
  $\rho \geq \lfloor(s+1)/2\rfloor \ge p+1.$
  Therefore, we can apply Lemma \ref{lem-low} to estimate $G_{\zeta \zeta}$ and $G_{\rho \rho}$ in \eqref{eq-beta} and \eqref{eq-estgamma} to get
  \begin{align}
    \sum_{i = 0}^s \beta_i G_{ii}  \leq& \beta_0 G_{00} + C \tau \sum_{\jj=0}^{p}H_{\jj \jj},\label{eq-beta-low}\\
    \sum_{i,j=0}^{s-1}\gamma_{ij}H_{ij} \leq& C\tau  \sum_{\jj=0}^{p}H_{\jj \jj}+C_{\gamma,-}\sum_{\jj =0}^{\rho-1}H_{\jj \jj}.\label{eq-estgamma-low}
  \end{align}
  To estimate \eqref{eq-estdelta}, note $\tilde{\delta}_{ij} = 0$ for $i + j \leq \sigma = s-1$. While for $i+j \geq s$, we have 
  \begin{equation*}
    \min_{i,j, i+j\geq s}\max\{i,j\} \geq \left\lceil\frac{i+j}{2}\right\rceil \geq  \left\lceil \frac{s}{2} \right\rceil \geq p+1.
  \end{equation*}
  Therefore, one can invoke \eqref{eq-lowXij} for each term in the summation \eqref{eq-estdelta}, which gives
  \begin{equation}\label{eq-estdelta-low}
    \begin{aligned}
      \sum_{i,j = 0}^{s-1}\tilde{\delta}_{ij}F_{ij} = \sum_{i+j\geq s, i,j\leq s-1 }\tilde{\delta}_{ij}F_{ij}
      \leq& C\veps^{-1}\tau^{2s-2p+1}G_{00} + \veps\sum_{\jj =0}^{p}H_{\jj \jj}.
    \end{aligned}
  \end{equation}
  With this we have bounded  all terms on the right hand side
  of~\eqref{eq-rsv-norm} except the last.
  For the last term in \eqref{eq-rsv-norm}, we can use Proposition \ref{prop-opest-norm} and Lemma \ref{lem-lowXk} to obtain
  \begin{equation}\label{eq-tail-low}
    \ip{M_0S_2v,M^{-1}M_1S_2v} \leq C\nm{X_s v}_{M_0}^2 = CG_{ss} \leq C\tau \sum_{\jj=0}^pH_{\jj \jj}.
  \end{equation}
  Using the bounds of \eqref{eq-beta-low}, \eqref{eq-estgamma-low},  \eqref{eq-estdelta-low} and \eqref{eq-tail-low} in \eqref{eq-rsv-norm}, we obtain 
  \begin{equation*}
    \begin{aligned}
      \nm{R_s v}_M^2\leq& \left(\beta_0 +C\veps^{-1}\tau^{2s-2p+1}\right)G_{00} +\left(C_{\gamma,-}+C\tau +\veps\right)\sum_{i=0}^{\rho-1}H_{ii}.
    \end{aligned}
  \end{equation*}
  Here we have used that $p \leq \rho-1.$ Note that $\beta_0 = 1$ and
  $G_{00} = \nm{v}_{M_0}^2$.  Since $C_{\gamma,-}<0$, we prove
  \eqref{eq-Rsvest-low} by taking $\tau$ and $\veps$ to be so small
  such that $C_{\gamma,-} + C\tau +\veps<0$.
\end{proof}

\begin{theorem}\label{thm-impweak}
	 Suppose $0<p\leq (s-1)/2$. Then there exists a constant $C$ such that when $\tau = r^{-1} \leq Ch^{1/(2s-2p)}$, we have 
	\begin{equation*}
	\nm{R_{r,s}v}_{M(1)} \leq (1+ Ch)\nm{v}_{M_0}, \quad \forall v\in V_h.
	\end{equation*}
	As a result, the SAT-DG method is weakly stable  under the $(1+1/(2s-2p))$-CFL condition $\Delta t \leq C h^{1+1/(2s-2p)}$. 
\end{theorem}

\begin{proof}
  Following along the lines of the proof of Theorem \ref{thm-satstab}, we have 
	\begin{equation*}
	\begin{aligned}
	\nm{R_{r,s}v}_{M(1)}\leq & \left(1 + C\tau^{2s-2p+1}\right)^{r}\nm{v}_{M_0}
	\leq (1+C\tau^{2s-2p+1}r)\nm{v}_{M_0}\\
	\leq & (1+C\tau^{2s-2p})\nm{v}_{M_0} \leq(1+Ch)\nm{v}_{M_0}. 
	\end{aligned}
	\end{equation*}
	Here we have used the fact $\tau = r^{-1} \leq C h^{1/(2s-2p)}$ in the last two inequalities. 
\end{proof}

\section{Illustration using linear advection}\label{sec-num}

In this section, we will use the one-dimensional linear advection 
\begin{equation}\label{eq-1dadv}
\partial_t u + \partial_x u = 0
\end{equation}
with the upwind DG discretization to illustrate that the estimates of Theorems \ref{thm-sat1}, \ref{thm-strongstab-onesub}, and \ref{thm-impstab-onesub} cannot
generally be improved. For the simple equation \eqref{eq-1dadv},  the associated matrices are of moderate sizes and many quantities can be evaluated analytically.

\subsection{Matrix form and the norm of the SAT operator} \label{sec-advimpl}

We consider a uniform mesh partition of mesh size $h$ of the
one-dimensional domain with the origin $x=0$ as a mesh point.  Setting
the pitch vertex $\vv$ to the spatial point $x = 0$, we pitch a tent
from $t = 0$, corresponding to $\vphi_\mathrm{bot} = 0$, over the
vertex patch $\Omega^\vv = I_0\cup I_1$. Here $I_0 = (-h,0]$ and
$I_1 = (0, h]$. We march forward in time to the point $(x,t) = (0,h),$
thus making a spacetime tent, and consider a subtent of it where
  pseudotime $\tau < 1$.  Since the causality condition is
  $\Delta t < h$ for this problem,
  $M(\tau)$ is invertible for $\tau < 1$.


Focusing on this single tent, we introduce a spatial basis of
Legendre polynomials $L_i(x)$, which is defined recursively as 
\begin{equation*}
L_0(x) = 1, \quad L_1(x) = x, \quand L_{i+1}(x) = 	\frac{2i+1}{i+1}xL_i(x) + \frac{i}{i+1}L_{i-1}(x), \quad \forall i \geq 2.
\end{equation*}
The piecewise polynomial basis over $\Omega^\vv$ for the DG space is defined using  normalized Legendre polynomials on $I_0$ and $I_1$, namely 
\begin{align*}
  b_{0i}(x)
  &
    = 	\sqrt{\frac{2i+1}{h}}\bar{b}_{0i}(x),
  &&\textrm{ with } \bar{b}_{0i}(x) = {L}_{i}\left(\frac{x+h/2}{h/2}\right)1_{[-h,0]}(x),
  \\
  b_{1i}(x)
  & = 	\sqrt{\frac{2i+1}{h}}\bar{b}_{1i}(x),
  && \textrm{ with } 
     \bar{b}_{1i}(x) = {L}_{i}\left(\frac{x-h/2}{h/2}\right)1_{[0,h]}(x).	
\end{align*}
In this basis, functions $w$ in $V_h$ are represented by their vector
$\mat w$ of coefficients in the basis expansion, i.e., 
\begin{equation*}
	w(x) = \sum_{i = 0}^p w_{0i} b_{0i}(x) + \sum_{i = 0}^p w_{1i} b_{1i}(x)
\end{equation*}
is represented as the \zs{vector
$\mat{w} = [w_{00}, \cdots w_{0p},w_{10},\cdots,w_{1p}]^\top$}. In the same basis, keeping the same block partitioning corresponding to the two spatial intervals, we have the following matrix representations of $M$, $M_0$, $M_1$ and $A$:
\begin{equation}
  \label{eq:matsMA}
	\mat{M}= \mat{M}_0 - \tau \mat{M}_1, \quad \mat{M}_0 = \mat{I}_{2p+2}, \quad 	
	\mat{M}_1 = \begin{bmatrix}
	\mat{I}_{p+1}&\mat{O}_{p+1}\\
	\mat{O}_{p+1}&-\mat{I}_{p+1}
\end{bmatrix}, \quad 	\mat{A} = \begin{bmatrix}
\mat{A}_{00}&\mat{A}_{01}\\
\mat{A}_{10}&\mat{A}_{11}	
\end{bmatrix},
\end{equation}
where $\mat{I}_\jj$ and $\mat{O}_\jj$ are the $\jj$th order identity and zero matrices, respectively, and 
\begin{equation*}
	\begin{split}
		&(\mat{A}_{00})_{ij} = \int_{-h}^{0}\left(x+h\right) b_{0j}\partial_xb_{0i}\dx - h b_{0j}(0)  b_{0i}(0),\\
		&(\mat{A}_{10})_{ij} = hb_{0j}(0)b_{1i}(0),
	\end{split}
\qquad
	\begin{split}
		&(\mat{A}_{01})_{ij} =  0,\\
		&(\mat{A}_{11})_{ij} = \int_{0}^{h}\left(h-x\right) b_{1j}\partial_xb_{1i}\dx.	
	\end{split}
\end{equation*}
One can show that $\mat{A}$ is  independent of $h$ (in accordance with Proposition~\ref{prop-opest-norm}).

The matrix representation of the SAT propagation operator can now be
written down using $\mat{X}_k$, defined recursively by
$\mat{X}_0 = \mat{I}_{2p+2}$ and
$\mat{X}_i = \tau \mat{M}_0^{-1}(\mat{A}+i\mat{M}_1)\mat{X}_{i-1}.$
The vector representation of $R_s w$ for any $w \in V_h$ equals
$\mat{R}_s \mat{w}$ where
\begin{equation*}
	\mat{R}_s = \sum_{k = 0}^{s-1} {(i!)^{-1}}\mat{X}_i + \mat{M}^{-1}\mat{M}_0({s!})^{-1}\mat{X}_s.
\end{equation*}
Following \cite[Section 6.1]{gopalakrishnan2020structure}, the norm  $\nm{R_s}_{L(M_0,M)}$ defined in \eqref{eq-Rsnorm}, can be computed by
\begin{equation}\label{eq-satnorm}
	\nm{R_s}_{L(M_0,M)} = \sup\{{|\lambda|^\hf}:0\neq \mat{w}\in \mathbb{R}^{2p+2}, \mat{R}_s^\top\mat{M}\mat{R}_s\mat{w} = \lambda \mat{M}_0 \mat{w} \}.
\end{equation}

\subsection{Numerical illustration of stability inequalities}

Since the matrix $\mat{R}_s$ depends only on $\tau$,  the operator norm in \eqref{eq-satnorm} is a one-variable function of $\tau$. We use  Mathematica$^\copyright$ for evaluation of \eqref{eq-satnorm} and conduct a Taylor expansion of $\nm{R_s}_{L(M_0,M)}$ with respect to $\tau$. The leading terms in these Taylor series are documented in Table \ref{tab-taylor} with different values of $s$ and $p$. Here is a summary of observations in the table:

\begin{enumerate}
\item For $p = 0$ and for any $s$, we observe that $\nm{R_s}_{L(M_0,M)} = 1$. 
\item For $0<p\leq (s-1)/2$, we observe that $\nm{R_s}_{L(M_0,M)} \leq  1 + C\tau^{2s-2p+1}$.
\item In general, we observe that $\nm{R_s}_{L(M_0,M)}\leq 1 + C\tau^{s+1}$. 
\end{enumerate}
These observations indicate that our analysis in the previous sections is sharp.
        
\begin{table}[htb!]
	\centering
	\begin{tabular}{r|r|r|r|r|r|r|r|r}
		\hline
		&$s = 1$&$s = 2$&$s=3$&$s = 4$&$s = 5$&$s = 6$&$s = 7$&$s = 8$\\
		\hline
		$p = 0$&$0$&$0$&$0$&$0$&$0$&$0$&$0$&$0$\\
		$p = 1$&$2\tau^2$&$0.08 \tau^3$&$0.25 \tau^5$&$0.25 \tau^7$&$0.25 \tau^9$&$0.25 \tau^{11}$&$0.25 \tau^{13}$&$0.25 \tau^{15}$\\
		$p = 2$&$13.32\tau^2$&$0.25\tau^3$&$0.08\tau^4$&$7.31\tau^5$&$3.19\tau^7$&$6\tau^9$&$9.91\tau^{11}$&$15\tau^{13}$\\
		$p = 3$&$46.20\tau^2$&$1.16\tau^3$&$0.09\tau^4$&$156.15\tau^5$&$7.72\tau^6$&$0.27\tau^7$&$27.19\tau^{9}$&$45.31\tau^{11}$\\
		$p = 4$&$117.66\tau^2$&$3.53\tau^3$&$0.10\tau^4$&$1511.49\tau^5$&$204.33\tau^6$&$2.81\tau^7$&$7.81\tau^{8}$&$21.47\tau^{9}$\\
		\hline
	\end{tabular}
\caption{Leading terms of Taylor expansions in $\nm{R_s}_{L(M_0,M)}-1$ with respect to $\tau$ for \eqref{eq-1dadv} as $\tau \ll 1$. For example, the entry $2\tau^2$ with $p = 1$ and $s = 1$ corresponds to $\nm{R_s}_{L(M_0,M)} = 1+ 2\tau^2 + \mathcal{O}(\tau^3)$. }\label{tab-taylor}
\end{table}

\section{Proofs of the lemmas}
\label{sec:proofs-sat}

\subsection{Proof of Lemma~\ref{lem-ibp}}

  When $j = i+1$, by the definition of $X_{i+1}$, we have 
  \[
    \begin{aligned}
      G_{ij}
      & = \ip{X_iv,X_{i+1}v}_{M_0} = \tau \ip{X_iv,\left(A+(i+1)M_1\right)X_iv}\\
      &=\tau \ip{X_iv,\left(\hf\left({A+A^\top+M_1}\right)+\left(i+\hf\right)M_1\right)X_iv}\\
      &=- \hf\snm{X_iv}_{\tau D}^2 +  \left(i+\hf\right)\ip{X_iv,X_iv}_{\tau M_1},
    \end{aligned}
  \]
which proves the first identity of the lemma.   In the case $j > i+1$, using the fact that $M_1$ is selfadjoint, we have
  \begin{equation*}
    \begin{aligned}
      G_{ij} & = \ip{X_iv,X_jv}_{M_0} = \tau\ip{X_i v, (A+jM_1)X_{j-1}v}\\
      &= \tau\ip{X_i v, \left(-(A+(i+1)M_1)^\top+(A+A^\top + M_1) + (i+j)M_1\right)X_{j-1}v}\\
      &= -\ip{\tau(A+(i+1)M_1)X_iv,X_{j-1} v}-\ip{X_iv,X_{j-1}v}_{\tau D} + (i+j)\ip{X_iv,X_{j-1}v}_{\tau M_1} \\
      &= -\ip{X_{i+1}v,X_{j-1}v}_{M_0}-\ip{X_iv,X_{j-1}v}_{\tau D} + (i+j)\ip{X_iv,X_{j-1}v}_{\tau M_1},
    \end{aligned}
  \end{equation*}
  which proves the second identity of the lemma.\hfill$\Box$

\subsection{Proof of  Lemma~\ref{lem-exps}}
  
We proceed to prove  Lemma~\ref{lem-exps} using  inductive applications of Lemma~\ref{lem-ibp}. It will be convenient to denote 
\[
  \nu_{lm} = \left\{
    \begin{array}{cc}
      1/2,& l = m\\
      1,& l\neq m
    \end{array}\right.
\]
and adopt the convention that $G_{\phi\phi}=0$ if $\phi$ is not an integer (so when $i+j$ is not even, the quantity $G_{\frac{i+j}{2},\frac{i+j}{2}}$ below vanishes).

\begin{lemma}
  \label{lem:Gij-identity}
  For any $j \ge i$,   we have
  \[
    \begin{aligned}
      G_{ij} = (-1)^{\frac{j-i}{2}} G_{\frac{i+j}{2},\frac{i+j}{2}} 
      & +
      \sum_{k = 0}^{\lfloor\frac{j-i-1}{2}\rfloor}(-1)^{k+1}
      \nu_{i+k,j-k-1}H_{i+k,{j-k-1}}
      \\
      & + \left(i+j\right)
      \sum_{k = 0}^{\lfloor\frac{j-i-1}{2}\rfloor}(-1)^{k}\nu_{i+k,j-k-1}F_{{i+k},{j-k-1}}.
    \end{aligned}
  \]
\end{lemma}
\begin{proof}
  The identity is trivial for the diagonal entries with $j=i.$
  For the superdiagonal entries, where $j=i+1$, the stated identity is
  the same as~\eqref{eq:Gi,j+1} of Lemma~\ref{lem-ibp}. When
  $j-i \ge 2$, the identity can be proved by applying
  Lemma~\ref{lem-ibp} recursively and formalizing by mathematical
  induction. If $j-i$ is even, then the recursion terminates in the
  obvious diagonal case. If $j-i$ is odd, then the recursion instead
  terminates in the superdiagonal case of~\eqref{eq:Gi,j+1}. In both
  cases we obtain the stated identity. 
\end{proof}

We will use the identity of Lemma~\ref{lem:Gij-identity} to expand
$\sum_{ij} \alpha_{ij} G_{ij}$ to prove Lemma~\ref{lem-exps}.  A few
more preparations on rearrangements of sums will be helpful for the
proof.

\begin{lemma}\label{lem-changesum0}
  For any numbers $\mu_{ij}$, the variable change $m = i+j$ and
  $\qq = i$ yields
	\begin{equation}
		\label{eq:changesumn1}
		\sum_{i,j = 0}^s \mu_{ij} = 	\sum_{m = 0}^{2s} \sum_{\qq = \max\{0,m-s\}}^{\min\{m,s\}} \mu_{\qq,m-\qq}.
	\end{equation}
	In particular, if $\mu_{ij} = 0$ when $i+j$ is odd, then the variable change $m = 2l$ in \eqref{eq:changesumn1} gives
	\begin{equation}
		\label{eq:changesum0}
		\sum_{i,j = 0}^s \mu_{ij} = 	\sum_{l = 0}^{s} \sum_{\qq = \max\{0,2l-s\}}^{\min\{2l,s\}} \mu_{\qq,2l-\qq}.
	\end{equation}
\end{lemma}
\begin{proof}
The sum over the discrete square region
  $0\leq i\leq s, 0\leq j \leq s,$ under the given variable change $m = i+j$ and
  $\qq = i$,
  becomes a sum over the discrete parallelogram region
  $P = \{ (\qq, m) \in \Z^2:\; 0\leq \qq \leq s,\; 0\leq m - \qq \leq
  s \},$ i.e.,
  \[
    \sum_{i,j = 0}^s \mu_{ij} = \sum_{(q, m) \in P}  \mu_{q, m-q}.
  \]
  It is easy to see (considering the boundaries of the parallelogram) that
  $P = \{ (q, m) \in \Z^2:\; 0 \le m \le 2s, \ \max(0,m-s)\leq \qq \leq
  \min(m,s)\},$ so \eqref{eq:changesumn1} follows. Finally, \eqref{eq:changesum0} can be obtained by dropping terms in \eqref{eq:changesumn1} with an odd $m$ and substituting in $m = 2l$.
\end{proof}

\begin{lemma}\label{lem-changesum}
  For any numbers $\mu_{ijk}$, the variable change $l = i+k$,
  $m=j-k-1$ and $\qq =i$ yields
  \begin{equation}
    \label{eq-changesum}
    \sum_{i,j = 0, \;j>i}^s\sum_{k=0}^{\lfloor \frac{j-i-1}{2}\rfloor}\mu_{ijk}
    \;= \sum_{l,m = 0\; m\geq l}^{s-1}\;
    \sum_{\qq=\max\{0,l+m+1-s\}}^l\mu_{\qq,l+m+1-\qq,l-\qq}.
  \end{equation}
\end{lemma}
\begin{proof}
  The left sum is over the region
  $\{(i, j,k ) \in \Z^3: 0\leq i<j\leq s,\; 0\leq k \leq
  ({j-i-1})/{2}\}$. The change of variable
  $i = \qq$, $j = l+m+1-\qq$ and $k = l-\qq$, obviously transforms the region  to
  \begin{equation*}
        T_1 = \left\{(q, l, m) \in \Z^3:\;
    0\leq \qq<l+m+1-\qq\leq s,\;  0\leq l-\qq \leq  \hf (l+m)-\qq\right\}. 
  \end{equation*}
  The sum on the right hand side of~\eqref{eq-changesum} is over
  \begin{equation}
    \label{eq:T2}
    T_2  = \{(q, l, m) \in \Z^3:\;
    0 \le l \le m \le s-1,\; 0 \le q, \; \; l+m+1-s \le q\le l\},
  \end{equation}
  so it is enough to show that $T_1 = T_2$.

  Let $(\qq, l, m) \in T_1$. Then since $0\leq \qq<l+m+1-\qq\leq s$,
  we have
  \begin{equation}
    \label{eq:3}
    0 \le q \quand   l+m+1 -s \le q,
  \end{equation}
  two inequalities needed for membership in $T_2$.  Moreover, since
  $0\leq l-\qq \leq \hf (l+m)-\qq$, we have $q \le l$ and $l \le m$,
  which together with~\eqref{eq:3} implies that $0 \le l$ and
  $l+m+1 -s \le q \le l$. The latter, in particular, implies
  $m \le s-1$. Thus, having obtained all the inequalities
  in~\eqref{eq:T2}, we conclude that $T_1 \subseteq T_2$. It is also
  easy to show that $T_2 \subseteq T_1,$ so $T_1=T_2.$
\end{proof}

\begin{proof}[Proof of Lemma~\ref{lem-exps}.]
  By Lemma~\ref{lem:Gij-identity} and the symmetry of $\alpha_{ij}$, 
  \[
    \sum_{i,j=0}^{s}\alpha_{ij} G_{ij} 
    = \sum_{i=0}^s \alpha_{ii} G_{ii}
    + 2\sum_{i,j=0, \; j>i}^{s}\alpha_{ij} G_{ij}
    = S_\beta + S_\gamma + S_\delta, 
  \]
  where
  \begin{gather*}
    S_\beta = \sum_{i, j=0}^s \alpha_{ij} (-1)^{\frac{j-i}{2}} G_{\frac{i+j}{2},\frac{i+j}{2}},
    \quad
    S_\gamma =       2\sum_{i,j=0, j>i}^s\alpha_{ij}\sum_{k = 0}^{\lfloor\frac{j-i-1}{2}\rfloor}(-1)^{k+1}\nu_{i+k,j-k-1}H_{i+k,j-k-1},
    \\
    S_\delta = 2 \left(i+j\right) \sum_{i,j=0,\; j>i}^s\alpha_{ij}
    \sum_{k = 0}^{\lfloor\frac{j-i-1}{2}\rfloor}(-1)^{k}\nu_{i+k,j-k-1}
    F_{i+k, j-k-1}.
  \end{gather*}
  By the variable change \eqref{eq:changesum0} in Lemma~\ref{lem-changesum0}, we have 
  \[
    S_\beta = \sum_{l = 0}^s \left(\sum_{\qq = \max\{0,2l-s\}}^{\min\{2l,s\}}\alpha_{\qq,2l-\qq}\left(-1\right)^{l-\qq}\right)
    G_{ll} = \sum_{l=0}^s \beta_l G_{ll}
  \]
  where $\beta_l$ is as defined in~\eqref{eq-abeta}.
  Next, apply the variable change of
  Lemma~\ref{lem-changesum} to $S_\gamma$ and $S_\delta$.
  Then
  \[
    \begin{aligned}
      S_\gamma
      & =
      2\sum_{l,m=0,m\geq l}^{s-1}\left(\sum_{\qq=\max\{0,l+m+1-s\}}^{\min\{l,m\}}(-1)^{\min\{l,m\}+1-\qq}\alpha_{\qq,l+m+1-\qq}\right)\nu_{lm} H_{lm}\\
      & =  \sum_{l,m =0}^{s-1}\left(\sum_{\qq=\max\{0,l+m+1-s\}}^{\min\{l,m\}}(-1)^{\min\{l,m\}+1-\qq}\alpha_{\qq,l+m+1-\qq}\right) H_{lm}
      = \sum_{l,m =0}^{s-1} \gamma_{lm} H_{lm}.
    \end{aligned}
  \]
  Here we have used the identity 
  $2\sum_{l,m = 0,\; m\geq l }^{s-1}\nu_{lm}\mu_{lm} = \sum_{l,m = 0}^{s-1}\mu_{lm},$
  which holds for all  $\mu_{lm}$ satisfying $\mu_{lm} = \mu_{ml}$.
  Similarly, for $S_\delta$, we have 
  \begin{equation*}
    \begin{aligned}
      S_\delta
      =& \sum_{l,m =0}^{s-1}\left(\sum_{\qq=\max\{0,l+m+1-s\}}^{\min\{l,m\}}(-1)^{\min\{l,m\}-\qq}\alpha_{\qq,l+m+1-\qq}\left(l+m+1\right)\right) F_{lm} =
      \sum_{l,m =0}^{s-1} \delta_{lm} F_{lm}.
    \end{aligned}
  \end{equation*}
  The sum of these expressions for $S_\beta$, $S_\gamma$, and
  $S_\delta$ proves~\eqref{eq-exp}.
\end{proof}

\subsection{Proof of Lemma~\ref{lem-bgd}}

In this proof, we shall use the following combinatorial identities. 
\begin{lemma}
  \label{lem:comb}
  \begin{gather}
    \sum_{\qq = 0}^{2i}\left({\qq!(2i-\qq)!}\right)^{-1}(-1)^{i-\qq} = 0, \quad \text{for any integer }i\geq 1, \label{eq-ci1}\\
    \sum_{\qq = 0}^{i}\binom{i+j+1}{\qq}(-1)^{i-\qq} =
    \binom{i+j}{i} \quad \text{for any integers }i,j\geq 0.\label{eq-ci2}
  \end{gather}	
\end{lemma} 
\begin{proof}
  To prove \eqref{eq-ci1}, we use the following binomial expansion
  for real $x$,
  \begin{equation*}
    x^i(1+x)^{2i} = x^i \sum_{\qq=0}^{2i}\binom{2i}{\qq}x^\qq = (2i)!\sum_{\qq = 0}^{2i}\left({\qq!(2i-\qq)!}\right)^{-1}x^{i+\qq}. 
  \end{equation*}
  The result follows by choosing   $x = -1$ and 
  replacing $(-1)^{i+q}$ with $(-1)^{i-q}$. 

  To prove \eqref{eq-ci2}, we will first show that given $l\geq 1$
  \begin{equation}\label{eq-ci3}
    \sum_{q=0}^i \binom{l}{q}(-1)^{i-q} = \binom{l-1}{i},\quad \text{for all }i \le l.
  \end{equation}
  Here a binomial coefficient
  $\binom{k}{i}$ is to be considered as zero when $i>k$, as happens for the  $i=l$ case above.  Obviously,
  \eqref{eq-ci3} holds for $l = 1$. To use induction on $l$, suppose
  \eqref{eq-ci3} holds for $l = \ii$ and for any $i\leq \ii$. Then
  using the identity
  $ \binom{\ii+1}{\qq} = \binom{\ii}{\qq}+\binom{\ii}{\qq-1}$ and the
  induction hypothesis, we have for $l=\ii+1$ and $i \le \ii$, 
  \begin{equation*}
    \begin{aligned}
      \sum_{\qq = 0}^{i} & \binom{\ii+1}{\qq}(-1)^{i-\qq} 
      =  \sum_{\qq = 0}^{i}\binom{\ii}{\qq}(-1)^{i-\qq}
      + \sum_{\qq = 1}^{i}\binom{\ii}{\qq-1}(-1)^{i-\qq} \\
      =& \sum_{\qq = 0}^{i}\binom{\ii}{\qq}(-1)^{i-\qq}  + \sum_{\qq = 0}^{i-1}\binom{\ii}{\qq}(-1)^{i-1-\qq} 
      = \binom{\ii-1}{i} + \binom{\ii-1}{i-1}
      = \binom{\ii}{i},
    \end{aligned}
  \end{equation*}
  i.e.,~\eqref{eq-ci3} holds for $l=\ii+1$ and $i \le \ii$.  The
  identity also holds for $l=\ii+1$ and $i=\ii+1$, as can be seen by
  choosing $x=-1$ in the binomial expansion of $(1+x)^{k+1}$. So we
  have shown that~\eqref{eq-ci3} holds for $l=\ii+1$ and
  $i \le \ii+1$.  Hence, by induction, \eqref{eq-ci3} holds for any
  $l \geq 1$. The identity \eqref{eq-ci2} follows by setting
  $l = i+j+1$ in~\eqref{eq-ci3}.
\end{proof}

\begin{proof}[Proof of  Lemma~\ref{lem-bgd}]
  From~\eqref{eq-abeta} of Lemma~\ref{lem-exps}, it is obvious that
  $\beta_0 = \alpha_{00} = 1$. When
  $1\leq i\leq s/2$, substituting $\alpha_{ij} = (i!j!)^{-1}$ into
  \eqref{eq-abeta} and using the \eqref{eq-ci1} of
  Lemma~\ref{lem:comb}, we obtain
  \begin{equation*}
    \beta_i = \sum_{\qq = 0}^{2i}\alpha_{\qq,2i-\qq}(-1)^{i-\qq} = 
    \sum_{\qq = 0}^{2i}(\qq!(2i-\qq)!)^{-1}(-1)^{i-\qq} = 0,
  \end{equation*}
  thus proving~\eqref{eq-ijbeta}. To prove \eqref{eq-ijgamma}, due to
  the symmetry of $\gamma_{ij}$, we proceed assuming without loss of
  generality that $j\geq i$. Then, for $i+j\leq s -1$, \eqref{eq-agamma} yields
  \begin{equation*}
    \begin{aligned}
      \gamma_{ij}
      & = \sum_{\qq=0}^{i} (-1)^{i+1-\qq}\alpha_{\qq,i+j+1-\qq}
      = \sum_{\qq=0}^{i}\frac{ (-1)^{i+1-\qq}  }{\qq!(i+j+1-\qq)!} \\
      &= -\left((i+j+1)!\right)^{-1}\sum_{\qq=0}^{i}\binom{i+j+1}{q}(-1)^{i-\qq} \\
      &= -\left((i+j+1)!\right)^{-1}\binom{i+j}{i}
      = -(i!j!(i+j+1))^{-1},
    \end{aligned}
  \end{equation*}
  where we have used \eqref{eq-ci2} of Lemma~\ref{lem:comb}. This
  proves~\eqref{eq-ijgamma}.  Finally, to prove \eqref{eq-ijdelta},
  note that~\eqref{eq-adelta} implies that
  $\delta_{ij} = -(i+j+1) \gamma_{ij}$. Hence the
  result $\delta_{ij}= (i!j!)^{-1}$ for $i+j\leq s-1$ follows immediately from the just established expression for $\gamma_{ij}$.
\end{proof}

\subsection{Proof of Lemma \ref{lem-zrsk}}\label{sec-proofzrsk}
\hfill

{\em Step~1.} It is immediate from \eqref{eq-ijbeta} and the
definition of $\zeta$ that $\zeta> s/2$. In fact
$\zeta = \lfloor s/2\rfloor + 1$. To see this, apply~\eqref{eq-abeta}
with $i=\lfloor s/2\rfloor + 1$, noting that $2i$ is either
$s+1$ or $s+2$.  Then applying \eqref{eq-ci1} to \eqref{eq-abeta} gives
\[
  \beta_{\lfloor s/2\rfloor + 1} =
  \left(
    \sum_{\qq = 0}^{2i} - \sum_{q=0}^{2i-(s+1)} 
    - \sum_{q=s+1}^{2i} 
  \right) 
  \alpha_{\qq,2i-\qq}(-1)^{i-\qq}  =
  - \left(\sum_{q=0}^{2i-(s+1)}+\sum_{\qq = s+1}^{2i} \right)
  \alpha_{\qq,2i-\qq}(-1)^{i-\qq}.
\]
With the variable change $l = 2i-q$ and the symmetry $\alpha_{ij} = (i!j!)^{-1}= \alpha_{ji}$, one can get 
\[
\beta_{\lfloor s/2\rfloor + 1} =
- \sum_{l=s+1}^{2i}\alpha_{2i-l,l}(-1)^{l-i} - \sum_{\qq = s+1}^{2i}
\alpha_{\qq,2i-\qq}(-1)^{i-\qq} = -2\sum_{\qq = s+1}^{2i}
\alpha_{\qq,2i-\qq}(-1)^{i-\qq}.
\]
This is a sum of one or two terms which can be easily verified to be nonzero. Hence $\beta_{\lfloor s/2\rfloor + 1}\neq 0,$  so $\zeta = \lfloor s/2\rfloor +1$.
	
{\em Step~2.} From~\eqref{eq-ijgamma}, we know
that in particular, for all $ 0\leq i,j \leq (s-1)/2,$
we have
$\gamma_{ij} = -(i!j!(i+j+1))^{-1}$. Let 
$\DD = \diag\left(0!,1!,2!,\cdots, \left\lfloor \frac{s-1}{2}\right\rfloor!\right).
$ Then 
\begin{equation}
  \label{eq:GammaMat}
  \Gamma_{\left\lfloor \frac{s+1}{2}\right\rfloor} =
  -\DD^{-1} \HH_{\left\lfloor \frac{s+1}{2}\right\rfloor}
  \DD^{-1}
\end{equation}
where $\HH_m$ denotes the $m \times m$ Hilbert matrix. Since Hilbert matrices are positive definite, the matrix in~\eqref{eq:GammaMat} is  negative definite. 
Hence $\rho \geq \lfloor(s+1)/2\rfloor$. 
	
{\em Step~3.} Since
$\tilde{\delta}_{ij} = \delta_{ij} - (i!j!)^{-1}$, using
\eqref{eq-ijdelta}, we have $\tilde{\delta}_{ij} = 0$ when
$i+j\leq s-1$. Hence we have $\sigma \geq s-1$.  Using \eqref{eq-adelta} and \eqref{eq-ci2}, it is easy to check that
$\tilde{\delta}_{ij}$ is nonzero when $i + j = s$. Hence
$\sigma = s-1$.
	
{\em Step~4.} Note that $2\zeta \geq s+1$, $2\rho + 1\geq s+1$, and $\sigma+2 = s+1$. As a result, we have $\kappa = \min(2\zeta,2\rho+1,\sigma+2) = \sigma + 2 = s+1$. 
\hfill$\Box$

\subsection{Proof of Lemma \ref{lem-est}}\label{sec-proofest}

By Lemma~\ref{lem-zrsk},
$
  \sum_{i = 0}^s \beta_i G_{ii} = \beta_0 G_{00} + 
  \sum_{i = \zeta}^s \beta_i G_{ii},
$
so
the first inequality of the lemma~\eqref{eq-beta} is immediately obtained
by applying \eqref{eq-xivest-norm} of
Proposition~\ref{prop-xivest-norm-rmd}.

To prove~\eqref{eq-estgamma}, since $\Gamma_\rho < 0 $ is negative
definite, there exists a constant $C_-<0$ such that
$\Gamma_{\rho} - C_- I_{\rho}<0$ remains negative definite. A simple
argument (see~\cite[Lemma 2.3]{sun2017rk4}) then proves that
$\sum_{i,j=0}^{\rho-1}[\Gamma_\rho - C_-I_\rho]_{ij}
\ip{X_iv,X_jv}_{\tau D} \le 0.$ Hence
\begin{equation}\label{eq-estgm1}
  \sum_{i,j = 0}^{\rho-1}\gamma_{ij} H_{ij}
  \leq C_-\sum_{\jj = 0}^{\rho -1} H_{ll}.
\end{equation}
The remaining summands on the left hand side of~\eqref{eq-estgamma}
involve indices with $\max(i,j)\geq \rho$.  By the Cauchy--Schwarz
inequality
$\left|\ip{X_iv,X_jv}_{\tau D}\right|\leq\snm{X_iv}_{\tau
  D}\snm{X_jv}_{\tau D},$ and if $i \ge \rho$, then
\eqref{eq-xivest-snm} of Proposition~\ref{prop-xivest-norm-rmd} yields
$\snm{X_iv}_{\tau D} \leq C\tau^{1/2}\nm{X_{\rho}v}_{M_0}.$ Thus
\begin{equation}\label{eq-estgm22}
  \left|\gamma_{ij} H_{ij}\right|
  \le C\tau G_{\rho\rho},
  \quad\text{ for $i \ge \rho$ and $j\ge \rho$}.
\end{equation}
In case only one of $i$ or $j$ is greater than or equal to $\rho$, say
$i\leq \rho-1$ and $j\geq \rho$ without loss of generality, then in
addition to the Cauchy--Schwarz inequality, we also apply the
inequality $ab \leq \veps a^2 + (4\veps)^{-1} b^2$, for any $0<\varepsilon<\varepsilon_0$ with $\varepsilon_0$ to be specified, to get
$\left|\ip{X_iv,X_jv}_{\tau D}\right|\leq  \snm{X_iv}_{\tau D}\snm{X_j v}_{\tau D} \leq \veps\snm{X_iv}_{\tau D}^2 + (4\veps)^{-1}\snm{X_jv}_{\tau D}^2.
$ Bounding the term with the larger index using 
\eqref{eq-xivest-snm}, we have
\begin{equation}\label{eq-estgm21}
  \left|H_{ji}\right| =\left|H_{ij}\right|\leq
  \veps H_{ii}+C\veps^{-1}\tau G_{\rho\rho},
  \quad\text{ for $i \le \rho-1$ and $j\ge \rho$}.
\end{equation}
Combining \eqref{eq-estgm1}, \eqref{eq-estgm22}, and \eqref{eq-estgm21}, it gives
\begin{equation*}
  \begin{aligned}
    \sum_{i,j=0}^{s-1}\gamma_{ij} H_{ij}
    \leq   C\left(1+\veps^{-1}\right)\tau G_{\rho\rho} +
    \left(C_-+C\veps\right)\sum_{\jj=0}^{\rho-1}H_{ll}.
  \end{aligned}
\end{equation*}
Choosing $\veps_0$ small enough so that $C_-+ C\veps_0\leq C_-/2$, we have
proven~\eqref{eq-estgamma} with
${C}_{\gamma,+} = C\left(1+\veps^{-1}\right) $ and
$C_{\gamma,-} = C_-/2$.

It only remains to prove~\eqref{eq-estdelta}.  In its left hand sum, by definition of $\sigma$ in Definition \ref{def:crit_ind},
only summands with indices in
$T_\sigma = \{(i, j) \in \Z^2: i+j>\sigma$ and $0\le i, j \le s-1\}$
are nontrivial. The summands can be bounded by
the estimates of Proposition~\ref{prop-opest-norm} and~\ref{prop-xivest-norm-rmd}
$F_{ij} \le C \tau
\nm{X_iv}_{M_0}\nm{X_jv}_{M_0} \le C \tau^{i+j+1} \| v \|_{M_0}^2.
$
Hence
\begin{align*}
  \sum_{i,j = 0}^{s-1}\tilde{\delta}_{ij}F_{ij}
  &
    =  \sum_{(i,j) \in T_\sigma}\tilde{\delta}_{ij}F_{ij}
    \le
    C \sum_{(i,j) \in T_\sigma} \tau^{i+j+1} G_{00}.
\end{align*}
Since $i+j+1 \ge \sigma+2$ for $(i,j) \in T_\sigma$, the inequality~\eqref{eq-estdelta} follows. \hfill$\Box$

\subsection{Proof of Lemma \ref{lem-lowXk}}\label{sec-prooflowXk}

We use induction. Note that $X_0 = I$ admits the described
form~\eqref{eq-lowXk}. Assuming that~\eqref{eq-lowXk} holds for
$i = \ii$, we need to prove that it holds for $i = \ii+1$. Subtracting
the recursive defining equation~\eqref{eq-Zi} of $Z_{\ii+1}$  from
that of $X_{\ii+1}$ (namely~\eqref{eq-Xi}), and using~\eqref{eq-lowA},
\[
      X_{\ii+1} - Z_{\ii+1}
      = \tau M_0^{-1}\left(A_1 + \ii M_1\right)(X_\ii - Z_\ii).
\]
\zs{Using the induction hypothesis $X_\ii   - Z_\ii= \delta^\ii  K^\ii$ and the definitions of $A_1$ and $M_1$,} we obtain,
for any $v \in V_h$,
\begin{align*}
  \left(X_{\ii+1} - Z_{\ii+1}\right)v
  &=  -\tau M_0^{-1}
    \left(\delta\, \divg_x\left(f(\delta^\ii  K^\ii v)\right)
    - \ii \, f(\delta^\ii  K^\ii v) \,\grad_x \delta\right)
  \\
    &=  -\tau M_0^{-1}
    \left(\delta\, \divg_x\left(\delta^\ii f(  K^\ii v)\right)
    - \ii \, \delta^\ii  f( K^\ii v) \,\grad_x \delta\right)
  \\
    &=  -\tau M_0^{-1} 
    \left(\delta^{\ii+1}\, \divg_xf(  K^\ii v) \right).
\end{align*}
\zs{Here we have used the fact that $f$ is homogeneous of degree 1 (recall \eqref{eq-gf}) in the second equality and the product rule for differentiation in the third equality.} Since $M_0$ acts point-wise (see \eqref{eq-lowM0M1}) the last
expression is the same as
$ \delta^{\ii+1} (-\tau M_0^{-1} \divg_xf )\circ K^\ii v =
\delta^{\ii+1} K^{\ii+1} v,$ thus establishing the formula~\eqref{eq-lowXk} for $i=\ii+1.$\hfill$\Box$

\subsection{Proof of Lemma \ref{lem-Zj}}\label{sec-proofZ}

We start by proving a preparatory bound on the norm of $\nm{L v}$.
\begin{lemma}\label{lem-L}
  For all $v \in V_h$,
  \[
    \nm{L v}\leq C\tau^{-\hf} \snm{v}_{\tau D} \leq C\nm{v}.
  \]
\end{lemma}
\begin{proof}
  The second inequality can be obtained by applying Proposition \ref{prop-xivest-norm-rmd}. We now prove the first inequality. Using inverse estimates and the fact $\nm{\delta}_{L^\infty}\leq Ch$, it can be seen that
\[
    \ip{\delta  \{w\},\{w\}}_{\Face^\vv} ^\hf \leq C\nm{w}\quand 
    \ip{ \delta  \lb w\rb, \lb w\rb }_{\Face^\vv} ^\hf \leq C\nm{w}.
\]
Hence using the Cauchy--Schwarz inequality and the fact that $\cD$ and $S$ are bounded, we get
\begin{align*}
	\ip{Lv,w} \leq & C\ip{ \delta \lb v\rb,\lb v\rb}_{\Face^\vv}^\hf \ip{ \delta   \{w\},  \{w\}}_{\Face^\vv}^\hf + C\ip{ \delta \lb v\rb, \lb v\rb}_{\Face^\vv}^\hf\ip{ \delta  \lb w\rb,  \lb w\rb}_{\Face^\vv}^\hf\\
	\leq & C\ip{ \delta \lb v\rb,\lb v\rb}_{\Face^\vv}^\hf\nm{w}.
\end{align*}
Taking $w = Lv$,
we deduce that 
\begin{equation}
  \label{eq:4nmLv}
\nm{Lv} \leq C\ip{ \delta \lb v\rb,\lb v\rb}^\hf_{\Face^\vv}.  
\end{equation}
By  \cite[Lemma 3.2]{DrakeGopalSchob21}, we have 
$	\snm{v}_{\tau D}^2 = 2\tau \ip{ \delta S\lb v\rb, \lb v\rb}_{\Face^\vv}.
$
Letting $\lambda>0$ denote  the smallest eigenvalue of the positive definite matrix $S$, we have
$\delta \lb v\rb\cdot \lb v\rb \le \lambda^{-1} \delta S\lb v\rb \cdot \lb v\rb$. Integrating and using~\eqref{eq:4nmLv},
\[
  C \nm{Lv}^2 \leq \ip{ \delta \lb v\rb,\lb v\rb}_{\Face^\vv}
  \le \frac{1}{\lambda } \ip{ \delta S\lb v\rb , \lb v\rb}
  =  \frac{1}{2\lambda \tau} |v|_{\tau D}^2.
\]
\end{proof}

\begin{proof}[Proof of Lemma \ref{lem-Zj}]
  It suffices to show that 
  \begin{equation}\label{eq-ZjI}
    \nm{Z_i v} \leq C\tau^\hf \sum_{\jj=0}^{i-1} \snm{X_\jj v}_{\tau D}, \quad \forall i \geq 0. 
  \end{equation}
  Indeed,~\eqref{eq-ZjI} implies
  \begin{equation*}
    \nm{Z_iv}_{M_0}^2\leq C\nm{Z_iv}^2 \leq C\left(\tau^\hf \sum_{\jj=0}^{i-1} \snm{X_\jj v}_{\tau D}\right)^2 \leq C\tau \sum_{\jj=0}^{i-1} \snm{X_\jj v}_{\tau D}^2 = C\tau \sum_{\jj=0}^{i-1} H_{\jj \jj},
  \end{equation*}
  thus completing the proof of \eqref{eq-Zj}.
	
  To prove \eqref{eq-ZjI}, we use induction on $i$. Since $Z_0 = 0$,
  the inequality~\eqref{eq-ZjI} certainly holds for the base case
  $i=0$.  Assume \eqref{eq-ZjI} holds for $i = \ii$. By the definition of $Z_{\ii+1}$ in~\eqref{eq-Zi} and the triangle inequality, we have 
\begin{align*}
\nm{Z_{\ii+1}v}& \le \zs{\tau \nm{M_0^{-1}\left(A+(\ii+1)M_1\right)}\nm{Z_\ii v}+\tau \nm{M_0^{-1}}\nm{LZ_\ii v} + \tau \nm{M_0^{-1}L X_{\ii }v}}\\
& \le  C\nm{Z_\ii v} + C\nm{L Z_\ii v}+ C\tau \nm{L X_{\ii }v}\\
& \le  C\nm{Z_\ii v}+C\tau^\hf \snm{X_\ii v}_{\tau D}\\
& \le  C\tau^\hf\sum_{\jj = 0}^{\ii}\snm{X_\jj v}_{\tau D}.
\end{align*}
Here we have applied Proposition \ref{prop-opest-norm} (and $\tau \leq 1$) in the second inequality, Lemma \ref{lem-L} in the third inequality, and the induction hypothesis in the last inequality. Therefore, \eqref{eq-ZjI} holds for $i = \ii+1$ and hence for all $i\geq 0$. 
\end{proof}

\subsection{Proof of Lemma~\ref{lem-low}}

We first prove \eqref{eq-lowXj}. Since $K^{p+1} v = 0$ for
$v \in P_p(K)$, we have $X_{p+1} = Z_{p+1}$. Therefore, using
Proposition \ref{prop-xivest-norm-rmd} and Lemma \ref{lem-Zj}, \zs{it can be shown that for any $i\geq p+1$,}
\begin{equation}\label{eq-normXjv}
  \begin{aligned}
    G_{ii} &= \nm{X_iv}_{M_0}^2\leq C\tau^{2i-2(p+1)} \nm{X_{p+1} v}_{M_0}^2\\
    &= C\tau^{2i-2(p+1)} \nm{Z_{p+1} v}_{M_0}^2\leq C \tau^{2i-2p-1} \sum_{\jj=0}^{p}H_{\jj \jj}\leq C \tau \sum_{\jj=0}^{p}H_{\jj \jj}.
  \end{aligned}
\end{equation}
Next, to prove~\eqref{eq-lowXij}, we apply the Cauchy--Schwarz inequality,
Proposition~\ref{prop-opest-norm}, and \eqref{eq-xivest-norm} to get
\begin{equation*}
  \begin{aligned}
    \left|F_{ij}\right| = \left|\ip{X_iv,X_jv}_{\tau M_1}\right|\leq C \tau \nm{X_iv}_{M_0}\nm{X_jv}_{M_0}\leq C\tau^{i+j-p}\nm{v}_{M_0}\nm{X_{p+1}v}_{M_0}.
  \end{aligned}
\end{equation*}
Invoking \eqref{eq-normXjv} with $i = p+1$,
\begin{equation}
  \left|F_{ij}\right|\leq C\tau^{i+j-p}G_{00}^\hf G_{p+1,p+1}^\hf \leq \left(C\tau^{i+j-p+\hf} G_{00}^\hf\right)\left(\sum_{\jj = 0}^{p}H_{\jj \jj}\right)^\hf,
\end{equation}
which yields \eqref{eq-lowXij} after applying the inequality $ab\leq (4\veps)^{-1}a^2 + \veps b^2$.\hfill$\Box$

\section{Conclusion}\label{sec-conclusion}

We have presented a systematic stability analysis of the SAT methods
for MTP schemes for solving linear hyperbolic equations.  We proved
the conjecture formulated in~\cite{DrakeGopalSchob21}, that the SAT
method is weakly stable under the $(1+1/s)$-CFL condition, is
true. The analysis in this paper generalizes the results in
\cite{sun2019strong} by including an affine linear time-dependent mass
matrix. Furthermore, improved stability estimates are obtained for
symmetric linear hyperbolic systems with piecewise constant coefficients and
with DG discretizations. With $P_0$-DG spatial discretization, the SAT
timestepping was proved to be strongly stable under the usual CFL
condition for any temporal order $s$. With $P_p$-DG spatial
discretization, the SAT scheme is weakly stable under the
$(1+{1/(2s-2p)})$-CFL condition when $0<p\leq (s-1)/2$. The estimates
are numerically verified to be sharp in each subtent for the
one-dimensional linear advection equation.  Finally, it is our hope
that the new understanding presented in our analysis will inspire
further ideas to improve numerical strategies for explicit
time-stepping on unstructured advancing fronts using tents. Of
particular interest is the development of a tent-based scheme that is
strongly stable under the usual CFL condition. Stabilization
techniques with artificial viscosity
\cite{sun2021enforcing,offner2020analysis} and the relaxation time
stepping methods \cite{ketcheson2019relaxation,ranocha2020relaxation}
may be promising avenues.

\section*{Acknowledgments}

The work of the first author was partially supported by the NSF grant DMS-1912779. The work of the second author was partially supported by the NSF grant DMS-2208391. We thank Dr. Jin Jin at John Hopkins University for helpful discussions that motivated the proof of Lemma \ref{lem-lowXk}.

\bibliography{refs}
\bibliographystyle{abbrv}

\end{document}